\documentclass[12pt, reqno]{amsart}

\usepackage{amsmath, amssymb, latexsym} % Math packages
\usepackage{graphicx}                   % For including graphics
\usepackage{url}                        % Formatting URLs
\usepackage{xcolor}                     % Text color
\usepackage{comment}                    % For commenting out blocks of text
\usepackage{mathrsfs}                   % Script font
\usepackage[a4paper, margin=1in]{geometry} % Layout adjustments

\allowdisplaybreaks
\newtheorem{thmA}{Theorem}  %alphabetic theorem counter: Theorem A, Theorem B, ...

\newcommand{\R}{{\mathbb R}}
\newcommand{\Q}{{\mathbb Q}}
\newcommand{\C}{{\mathbb C}}

\newcommand{\Z}{{\mathbb Z}}

\def\F {F^{(k)}}

\theoremstyle{plain}
\numberwithin{equation}{section}
\newtheorem{thm}{Theorem}[section]
\newtheorem{lemma}[thm]{Lemma}

\newtheorem{proposition}[thm]{Proposition}

\begin{document}

\title[Diophantine equations over the generalized Fibonacci sequences]{Diophantine equations over the generalized Fibonacci sequences: exploring sums of powers}

\author[R. Alvarenga]{Roberto Alvarenga}
\address{São Paulo State University (UNESP)\\
               São José do Rio Preto-SP\\
               15054-000, Brazil}
\email{roberto.alvarenga@unesp.br}

\author[A.P. Chaves]{Ana Paula Chaves}
\address{Instituto de Matemática e Estatística\\
               Universidade Federal de Goiás\\
               Goiânia, GO\\
               74001-970, Brazil}
\email{apchaves@ufg.br}

\author[M. E. Ramos]{Maria Eduarda Ramos}
\address{Departamento de Matemática\\
               Universidade Federal de Minas Gerais\\
               Belo Horizonte, MG\\
               31270-901, Brazil}
\email{madu-ramos@ufmg.br}

\author[M. Silva]{Matheus Silva}
\address{Instituto De Ciências Matemáticas e de Computação\\
               Universidade de São Paulo\\
               São Carlos, SP\\
               13566-590, Brazil}
\email{matheussilva1@usp.br}

\author[M. Sosa]{Marcos Sosa}
\address{Universidade Federal da Integração Latino-Americana\\
                Foz do Iguaçu, PR\\
               85870-650, Brazil}
\email{mes.garcete.2020@aluno.unila.edu.br}
\thanks{}

\begin{abstract}
Let \((F_n)_{n}\) be the classical Fibonacci sequence. It is well-known that it satisfies \(F_{n}^2 + F_{n+1}^2 = F_{2n+1}\). In this study, we explore generalizations of this Diophantine equation in several directions. First, we solve the Diophantine equation 
$\big(F_{n}^{(k)}\big)^2 + \big(F_{n+d}^{(k)}\big)^2 = F_{m}^{(k)}
$
over the \(k\)-generalized Fibonacci numbers for every \(k \geq 2\), generalizing \cite{chaves-marques-14}. Next, we solve 
$
F_{n}^{s} + F_{n+d}^{s} = F_m
$
over the Fibonacci numbers for every \(s \geq 2\), generalizing \cite{luca-oyono-11}. Finally, we solve the Diophantine equation 
$
F_{n}^s + \cdots + F_{n+d}^s = F_m
$
for \(d+1 < n\) and \(s \geq 2\).
   
\end{abstract}

\setcounter{tocdepth}{1}
\maketitle
\tableofcontents

%%%%%%%%%%%%%%%%%%%%%%%%%%%%%%%%%%%%%%%%%%%%%%%%%%%%%%%%%%%%%%
%%%%%%%%%%%%%%%%%%%%%%%%%%%%%%%%%%%%%%%%%%%%%%%%%%%%%%%%%%%%%%
%%%%%%%%%%%%%%%%%%%%%%%%%%%%%%%%%%%%%%%%%%%%%%%%%%%%%%%%%%%%%%
%%%%%%%%%%%%%%%%%%%%%%%%%%%%%%%%%%%%%%%%%%%%%%%%%%%%%%%%%%%%%%
%%%%%%%%%%%%%%%%%%%%%%%%%%%%%%%%%%%%%%%%%%%%%%%%%%%%%%%%%%%%%%
%%%%%%%%%%%%%%%%%%%%%%%%%%%%%%%%%%%%%%%%%%%%%%%%%%%%%%%%%%%%%%

\section{Introduction}

%%% PART I - Introd to (k-generalized) Fibonnaci

Let \((F_n)_{n}\) be the Fibonacci sequence defined by \(F_{n+2} = F_{n+1} + F_n\) for \(n \geq 0\), with initial conditions \(F_0 = 0\) and \(F_1 = 1\). Fibonacci numbers are renowned for their fascinating properties and remarkable connections to various fields, including natural structures, architecture, engineering, technology, computing, and more. For an overview of the history of the Fibonacci sequence, properties, and intriguing applications and generalizations, see \cite{vorobiev-02}. Let $k \in \Z$, $k \geq 2$. Among several generalizations of the Fibonacci numbers, we consider the \emph{\(k\)-generalized Fibonacci sequence}  \((\F_n)_{n \geq -(k-2)}\) (also known as the \emph{\(k\)-step Fibonacci sequence} or \emph{\(k\)-bonacci sequence}), defined for \(n \geq 0\) by
\[
\F_{n+k} = \F_{n+k-1} + \F_{n+k-2} + \cdots + \F_{n},
\]  
and with initial values \(\F_{-k+2} := \F_{-k+1} := \cdots := \F_0 := 0\) and \(\F_1 := 1\).  We observe that for \(k = 2\) in the above definition, we recover the well-known Fibonacci numbers, i.e., \(F_n^{(2)} = F_n\). 
These \(k\)-generalized Fibonacci sequences have been extensively investigated in the literature over the last three decades. According to \cite{kessler-04}, these numbers appear in probability theory and certain sorting algorithms. In \cite{altassan-alan-23}, the authors introduce the \(k\)-generalized \emph{tiny golden angles} using the roots of the characteristic polynomial of the \(k\)-generalized Fibonacci sequence. The authors assert that these angles could enhance the accuracy of MRI scans by optimizing the distribution of sampling points, thereby resulting in clearer images.

%%% PART II - Introd to Diophantine Eq

On the other hand, Diophantine equations is a major topic in mathematics. It addresses solutions of polynomial equations over non-algebraically closed fields or, more generally, over (a subset of) any ring. Problems in this area of research are often simple to state but require techniques from several fields, such as commutative algebra, algebraic geometry, and algebraic and analytic number theory, to achieve resolution. A notable example is Fermat's Last Theorem, proved by Andrew Wiles in \cite{wiles-95}, which asserts that the Diophantine equation \(X^n + Y^n = Z^n\) admits no nontrivial integer solutions for \(n > 2\).

%%% PART III - Our main results

Diophantine equations involving Fibonacci numbers and their generalizations have been the focus of numerous research articles for at least the past two decades. For instance, we refer to \cite{bravo-luca-13}, \cite{chaves-marques-15}, \cite{ddamulira-gomez-luca-18}, and \cite{gomez-gomez-luca-24} for some recent works in this direction. In the seminal work \cite{bugeaud-mignotte-siksek-06}, Bugeaud, Mignotte, and Siksek utilized tools from the proof of Fermat's Last Theorem and linear forms in logarithms to solve the Diophantine equation $F_n=x^y$, where $x,y \in \Z$ with $x,y > 0$. Thereby, identifying all perfect powers within the Fibonacci sequence.

In this article, we investigate Diophantine equations over the $k$-generalized Fibonacci numbers. 
Our starting point is the following well-known identity, which is satisfied by the Fibonacci numbers
\begin{equation}\tag{$0$} \label{eq-classic}
    F_{n}^2 + F_{n+1}^2 = F_{2n+1}.
\end{equation}
We propose the following generalizations of this identity.

%%% MAIN THEOREM 1 and 1' 

The first natural question that arises is:  what is the behavior of the sum of squares of two $k$-generalized Fibonacci numbers, not necessarily consecutive or distinct? Namely, we aim to investigate the solutions of 
\begin{equation}\tag{I} \label{eq-firstgoal} 
    \big(\F_{n} \big)^2 + \big( \F_{n+d} \big)^2 = \F_{m}
\end{equation}
for every $k,d \in \Z$ with $k \geq 2$ and $d \geq 0$. We observe that when $k=2$ and $d=1$ we return to identity \eqref{eq-classic}. 
Geometrically, solving \eqref{eq-firstgoal} is equivalent to searching for the \textit{Fibonacci rational points} (i.e., integer solutions that are also Fibonacci numbers) of the following Diophantine equation
\[ X^2 + Y^2 = Z.\]
Our first main theorem, which is Theorem \ref{propfibquad}, solves \eqref{eq-firstgoal} for $k=2$:  

\begin{thmA}  The solutions of the Diophantine equation \eqref{eq-firstgoal}, for $k=2$,
    are 
    \[ (n,d,m) \in \big\{(1,2,5), (1,0,3), (2,0,3), (3,0,6), (n,1,2n+1) \big\},\]
where $d,n \in \Z$, with $n \geq 1$ and $d \geq 0$. 
\end{thmA}

For $k \geq 3$, in  \cite{chaves-marques-14}, Chaves and Marques provided a partial solution to \eqref{eq-firstgoal}. Namely, they considered the sum of squares of two consecutive $k$-generalized Fibonacci numbers, i.e., $(\F_n)^2 + (\F_{n+1})^2$, and showed that it does not belong to the same sequence $(\F_m)_m$, if $n >1$. The following is Theorem \ref{thm-main1}, which leads to a complete solution for \eqref{eq-firstgoal} and generalizes \cite[Thm. 1.1]{chaves-marques-14}. 
 
\begin{thmA}  Let $n,k,d \in \Z$ with $n\geq 1$, $k \geq 3$ and $d \geq 0$. The Diophantine equation \eqref{eq-firstgoal} 
has a solution if, and only if, either 
\[ (n,d,m) = (1,0,3)  \quad \text{ or } \quad 
(n,d,m) = (1,1,3) \quad \text{ or } \quad 
(n,d,m) = (a,0,2a-1), \] 
where $2 \leq a \leq \lfloor \frac{k+2}{2}\rfloor$. 
\end{thmA}

%%% MAIN THEOREM 2

Let $s \in \mathbb{Z}$, with $s \geq 2$. We now investigate the following question: what is the behavior of the sum of two $s$-powers of Fibonacci numbers, not necessarily consecutive or distinct? Namely, we aim to generalize identity \eqref{eq-classic} by considering the sum of two $s$-powers of arbitrary Fibonacci numbers, instead of the squares of two consecutive Fibonacci numbers, i.e., we aim to investigate the solutions of 
\begin{equation} \tag{II} \label{eq-secondgoal}
    F_n^s + F_{n+d}^s = F_m,
\end{equation}
for $d \in \Z$, $d \geq 0$. Geometrically, solving \eqref{eq-secondgoal} is equivalent to searching for the \textit{Fibonacci rational points} of the following Diophantine equation
\[ X^s + Y^s = Z.\]
If $d=1$,  Luca and Oyono prove in \cite{luca-oyono-11} that equation \eqref{eq-secondgoal} has no solution for $n \geq 2$ and $s \geq 3$. The following is Theorem \ref{thm-main2}, which yields a complete solution for \eqref{eq-secondgoal} and generalizes \cite[Thm. 1]{luca-oyono-11}.

\begin{thmA}  
Let $ d, n, m, s \in \Z$, with $ s \geq 2 $. The Diophantine equation  \eqref{eq-secondgoal}  
has no solution for $(n,d) \not\in \{(1,2),(1,1)\}$. If $s=2$, $n=1$ and $d=2$ provide the unique nontrivial solution. Moreover, if $s \geq 3$, the unique solution is $n = d = 1$.
\end{thmA}

%%% MAIN THEOREM 3 
Finally, we return to identity \eqref{eq-classic} and pose the following question: what is the behavior of the sum of \(s\)-powers of arbitrary consecutive Fibonacci numbers? Namely, we aim to investigate the solutions of
\begin{equation} \tag{III} \label{eq-thirdgoal}
    F_n^s + F_{n+1}^s + \cdots + F_{n+d}^s = F_m,
\end{equation}  
where $s,d \in \Z$ with $s \geq 3$ and $d \geq 2$. Geometrically, solving \eqref{eq-thirdgoal} is equivalent to searching for consecutive \textit{Fibonacci rational points} of the following Diophantine equation 
\[ X_1^s + \cdots + X_d^s = Y. \]  
In this context, our main result is Theorem \ref{thm-main3}, which we state below. It addresses this problem with a minor restriction on the number of terms involved.

\begin{thmA} 
    Let $m,n,s,d \in \Z$ be such that $n\geq3$, $s\geq3$, and $d\geq2$. 
    If $d+1<n$, then the Diophantine equation \eqref{eq-thirdgoal} 
has no solution. 
\end{thmA}

It is worth mentioning that solving most of these Diophantine equations often requires the use of computational tools. These tools are crucial, for instance, for applying the Dujella-Pethő Lemma (Lemma \ref{lemma-dujella}), a key result for refining upper bounds for the variables involved. Furthermore, once the variables have been bounded, the use of computational tools is essential for verifying these Diophantine equations in the remaining finite cases. A brief explanation of the use of these computational tools is provided in the appendix of the paper.

%%%%%%%%%%%%%%%%%%%%%%%%%%%%%%%%%%%%%%%%%%%%%%%%%%%%%%%%%%%%%%
%%%%%%%%%%%%%%%%%%%%%%%%%%%%%%%%%%%%%%%%%%%%%%%%%%%%%%%%%%%%%%
%%%%%%%%%%%%%%%%%%%%%%%%%%%%%%%%%%%%%%%%%%%%%%%%%%%%%%%%%%%%%%
%%%%%%%%%%%%%%%%%%%%%%%%%%%%%%%%%%%%%%%%%%%%%%%%%%%%%%%%%%%%%%
%%%%%%%%%%%%%%%%%%%%%%%%%%%%%%%%%%%%%%%%%%%%%%%%%%%%%%%%%%%%%%
%%%%%%%%%%%%%%%%%%%%%%%%%%%%%%%%%%%%%%%%%%%%%%%%%%%%%%%%%%%%%%

\section{Background} 

In this section, we gather basic facts about $k$-generalized Fibonacci numbers and standard tools that we will need in the following sections. Throughout the article, we assume $d,k,m,n, s \in \Z$, with $d \geq 0$, $k \geq 2$ and $s \geq 2$.

\subsection*{Elementary facts.} Let 
\[  f_k(T) := T^k - T^{k-1} - \cdots - T - 1 \in \Z[T],\]
be the characteristic polynomial of $(F_{n}^{(k)})_n$.  It is well known that $f_k(T)$ is an irreducible polynomial over $\Q[T]$ with simple roots, cf.\ either \cite{miles-60} or \cite[Cor. 3.4 and 3.8]{wolfram-98}. Moreover, its roots $ \alpha_1, \ldots, \alpha_k \in \C$, can be ordered such that
\[  3^{-k} < |\alpha_k | \leq \cdots  \leq |\alpha_2 | < 1 < |\alpha_1 |   \]
where actually $\alpha_1 \in \R$, cf.\ \cite{miles-60}. The root $\alpha_1 \in \R$ is called the \emph{dominant root} of either $f_k(T)$ or $(F_{n}^{(k)})_n$. If $k=2$,  we denote $\alpha := \alpha_1$ and $\beta := \alpha_2$. In this case, $\alpha = (1+ \sqrt{5})/2$ and $\beta = -\alpha^{-1}$. 

In the following lemma, we summarize several useful facts.

\begin{lemma} \label{kfibbasic} With the notation introduced above. 
    \begin{enumerate}
       \item[(i)] The inequality 
        $ \alpha_{1}^{n-2} \leq \F_n \leq \alpha_{1}^{n-1}$,
        holds for every $n \geq 1$. 
        
        \item[(ii)] $2(1-2^{-k}) < \alpha_1 < 2$.
        
        \item[(iii)] If $2 \leq n \leq k+1$, then $\F_n = 2^{n-2}$. 

        \item[(iv)] The identities  $ F_n = (\alpha^n - \beta^n)/\sqrt{5}$ and $\alpha^n= \alpha F_n + F_{n-1}$ hold for every $n \geq 0$. 
       
       \item[(v)] The inequality $F_n/F_{n+d} < (2/3)^d$ holds for every $n \geq 2$ and $d \geq 1$.    
\end{enumerate}
\end{lemma}

\begin{proof}
   The proof of these facts can be found in several references. For instance, the first item is given by \cite[Lemma 1]{bravo-luca-publ-13}, the second item by \cite[Lemma 3.6]{wolfram-98}, and the third item follows from \cite{bravo-luca-JNT-13}. The first identity in the fourth item is the well-known \emph{Binet's formula}, which also has an analogue version for \(k > 2\) (see, e.g., \cite{dresden-du-14}). The second identity in the fourth item can be proved by simple induction. The final item follows from the fact that \( F_n / F_{n+1} < 2/3 \).
\end{proof}

\subsection*{Matveev's Lemma} This Lemma is a standard tool used to provide a lower bound for linear forms in logarithms, \emph{à la Baker}. Since we will need to calculate the terms and constants that appear in the lemma, we state it below in its complete version.  

\begin{lemma}[Matveev] \label{lemma-matveev} Let $\gamma_1, \ldots, \gamma_n \in \R_{> 0} \cap \overline{\Q} \setminus \{0,1\}$ and $b_1, \ldots, b_n \in \Q$. If 
$ \gamma_{1}^{b_1}  \cdots  \gamma_{n}^{b_n} \neq 1,$
then 
\[ (eB)^{-\lambda} <   \big| \gamma_{1}^{b_1}  \cdots  \gamma_{n}^{b_n} - 1 \big|, \]
where $B := \max\{\big|b_1\big|, \ldots, \big|b_n\big|\}$ and $\lambda \in \R$ is described as follows.
Let $K$ be a degree $\ell$ number field containing  
$\gamma_1, \ldots, \gamma_n$. We define
\[ C_{n,\ell} := 1.4\cdot 30^{n+3} \cdot n^{4.5} \cdot \ell^2 (1+ \log(\ell)).\]
Let $A_1, \ldots, A_n \in \R$ be such that 
 $A_i \geq \max \big\{\ell h(\gamma_i), \big|\log(\gamma_i)\big|, 0.16 \big\}$,
where, for $\gamma \in \overline{\Q}$ of degree $d$:
\[ h(\gamma) :=  \frac{1}{d} \left( \log(a_0) + \sum_{i=1}^{d} \log(\max\{\gamma^{(i)},1\}) \right)\]
is the logarithmic height of $\gamma$ and 
\[ p(T) := a_0 \prod_{i=1}^{d} (T - \gamma^{(i)}) \in \Z[T]\]
is the minimal primitive polynomial of $\gamma$ having positive leading coeficient. Therefore, 
\[ \lambda := C_{n,\ell} \prod_{i=1}^n A_i  .\]
In particular, $\lambda $ is an effectively computable constant depending only on $\gamma_1, \ldots, \gamma_n.$
\end{lemma}

\begin{proof}
    See \cite{baker-75} and \cite{matveev-00} for an explicit description of $\lambda\in \R$. We also refer to \cite[Thm. 9.4]{bugeaud-mignotte-siksek-06}.
    \end{proof}
 
%%%%%%%%%%%%%%%%%%%%%%%%%%%%%%%%%%%%%%%%%%%% 

\subsection*{Continued fractions} Let $x \in \R$. As usual, we denote the simple continued fraction representation of $x$ by $ [a_0; a_1, a_2,\ldots ]$. For the reader's convenience, we recall the following classical result. It provides a key tool to refine the bounds obtained for the variables involved in the Diophantine equations.

\begin{lemma}[Legendre] \label{lemma-legendre}
Let $\gamma \in \R \setminus \Q$ and $p/q \in \Q \setminus \{0\}$. If
\[\left| \gamma - \frac{p}{q} \right| < \frac{1}{2q^2},\]
then $p/q$ is a convergent of the simple continued fraction representation of $\gamma$.
\end{lemma} 

\begin{proof} See \cite[Thm. 3.18]{gugu-18}.
\end{proof}

\subsection*{Dujella-Pethő's Lemma} This lemma provides an important criterion for the nonexistence of solutions to certain Diophantine inequalities involving the approximation of irrational numbers by continued fractions. In its statement, for \(x \in \mathbb{R}\), the norm \(\lVert x \rVert\) is taken as the distance from \(x\) to the nearest integer, i.e., \(\displaystyle \lVert x \rVert := \min_{n \in \mathbb{Z}} |x - n|\).

\begin{lemma}[Dujella-Pethő] \label{lemma-dujella}
Let $\ell \in \Z_{> 0}$, $\mu \in \R$ and $\gamma \in \R\setminus \Q$. Let \( p/q \in \Q \) be a convergent of the continued fraction expansion of \(\gamma\) such that \(q > 6\ell\). 
Let \(\epsilon := \lVert \mu q \rVert - \ell \lVert  \gamma q \rVert\). If \(\epsilon > 0\), then the inequality  
\[0 < n\gamma - m + \mu < AB^{-n}\]
has no solution for \(n, m \in \Z_{>0} \) with  
\[\frac{\log(Aq/\epsilon)}{\log(B)} \leq n \leq \ell.\]
\end{lemma}

\begin{proof} See \cite{dujella-petho-98}.
\end{proof}

%%%%%%%%%%%%%%%%%%%%%%%%%%%%%%%%%%%%%%%%%%%%%%%%%%%%%%%%%%%%%%
%%%%%%%%%%%%%%%%%%%%%%%%%%%%%%%%%%%%%%%%%%%%%%%%%%%%%%%%%%%%%%
%%%%%%%%%%%%%%%%%%%%%%%%%%%%%%%%%%%%%%%%%%%%%%%%%%%%%%%%%%%%%%
%%%%%%%%%%%%%%%%%%%%%%%%%%%%%%%%%%%%%%%%%%%%%%%%%%%%%%%%%%%%%%
%%%%%%%%%%%%%%%%%%%%%%%%%%%%%%%%%%%%%%%%%%%%%%%%%%%%%%%%%%%%%%
%%%%%%%%%%%%%%%%%%%%%%%%%%%%%%%%%%%%%%%%%%%%%%%%%%%%%%%%%%%%%%

\section{Sum of two squares of $k$-generalized Fibonacci numbers}

As previously mentioned, the fact that the sum of the squares of two consecutive Fibonacci numbers belongs to the Fibonacci sequence has sparked the exploration of various problems. This has urged the development of both elementary and sophisticated tools. This section addresses a natural question that arises from this fact. Namely, in this section, we solve the following Diophantine equation
\[ \big( \F_n \big)^2 + \big( \F_{n+d} \big)^2 = \F_m \]
for $k \geq 2$ and $d \geq 1$.

We first introduce a key lemma that slightly improves  Lemma \ref{kfibbasic}, item (i), for $k=2$.  

\begin{lemma} \label{desigfibonacci}
   The inequality
    \[
    \alpha^{n-7/4} < F_n < \alpha^{n-3/2} 
    \]
    holds for all $n \geq 3$. 
\end{lemma}

\begin{proof} For $n=3$ and $4$, 
\[\alpha^{3-\frac{7}{4}}<1.83 <  F_3 < 2.05<\alpha^{3-\frac{3}{2}} \ \mbox{ and } \  \alpha^{4-\frac{7}{4}} < 2.96 <  F_4 < 3.33 < \alpha^{4-\frac{3}{2}}.\]
Assuming that the inequality holds for every $3 \leq k \leq n$, yields
    \[\alpha^{n-1-\frac{7}{4}} + \alpha^{n-\frac{7}{4}}< F_{n-1}+F_n < \alpha^{n-1-\frac{3}{2}} + \alpha^{n-\frac{3}{2}}.\]
Hence,  
\[\alpha^{n-\frac{7}{4}}(1+\alpha^{-1})< F_{n+1} < \alpha^{n-\frac{3}{2}}(1+\alpha^{-1}).\]
Since $\alpha^2=1+\alpha$, the lemma follows.
\end{proof}

By using the previous lemma and a few clever algebraic manipulations, we can exhibit all Fibonacci numbers that can be expressed as the sum of the squares of two Fibonacci numbers. This solves the problem proposed at the beginning of this section for $k=2$.

\begin{thm}
    \label{propfibquad}  The solutions for the Diophantine equation
    \begin{equation} \label{quadfib}
      F_{n}^2 + F_{n+d}^2 = F_m ,
    \end{equation}
    for $n\geq 1$, are $(n,d,m) \in \{(1,2,5), (1,0,3), (2,0,3), (3,0,6), (n,1,2n+1)\}$.
\end{thm} 

\begin{proof} First we deal with the case $d=0$, which gives $2F_n^2 = F_m$. By Carmichael's theorem, we must have $m \leq 12$, and a quick search yields $(n,d,m) = (1,0,3), (2,0,3)$ or $(3,0,6)$. For $n=1$, equation \eqref{quadfib} becomes $1+ F_{d+1}^2=F_m$. If $d\geq3$, we obtain
\[1+ F_{d+1}^2 < F_d^2 + F_{d+1}^2 = F_{2d+1}.\]
On the other hand, from Lemma \ref{desigfibonacci} 
\[ F_{2d}<(\alpha^{d + 1 -\frac{7}{4}})^2 < F_{d+1}^2 < 1 + F_{d+1}^2. 
\]
Thus, $1 + F_{d+1}^2$ falls between two consecutive Fibonacci numbers and cannot belong to the sequence, leading to no solutions in this case. If $n=1$ and $d=2$, we have $(n,d,m)=(1,2,5)$. 
Analogously, for $n=2$ and $d\geq2$, the same argument ensures that
\[F_{2d+2} < 1 +F_{d+2}^2 < F_{2d+3},\]
then the given case leads to no solutions. Finally, if $n>2$, first notice that
    \[F_n^2+ F_{n+d}^2 < F_{n+d-1}^2 + F_{n+d}^2 = F_{2n+2d-1}.\]
On the other hand,
\[
F_{2n+2d-1} = F_{n+d-1}^2 + F_{n+d}^2 < F^2_n  + F_{n+d-2}^2 + F_{n+d-1}^2 +  F^2_{n+d} =  F^2_n + F_{2n+2d-3} +  F^2_{n+d},
\]
therefore,
\[
F_{2n-2d-2} = F_{2n+2d-1} - F_{2n+2d-3} < F^2_n + F^2_{n+d},
\]
which combined with our previous findings this yields
\[
F_{2n-2d-2} < F^2_n + F^2_{n+d} < F_{2n+2d-1},
\]
guaranteeing that $F^2_n + F^2_{n+d}$ do not belong to $(F_n)_n$.
\end{proof}

%%%%%%%%%%%%%%%%%%%%%%%%%%%%%%%%%%%%%%%%%%%%%%%%%
The following result is central to the proof of the main result in this section. It provides an alternative way to express a $k$-generalized Fibonacci number in terms of its preceding terms, allowing us to obtain some key inequalities.

\begin{lemma} \label{lemkbona} The identity 
    \[F_n^{(k)} = \sum_{j=0}^{k-1}F_{\ell-j}^{(k)} \left( \sum_{i=0}^{k-1-j}F_{n-\ell-i}^{(k)} \right),\]
holds for every  $\ell \in \Z$ such that  $1\leq \ell \leq n-1$, with $n \geq 1$. 
\end{lemma}

\begin{proof}
    We first verify the case $\ell=n-1$ and thus use the descent method to show that if the assertion holds for $1< \ell \leq n-1$, then it holds for $\ell-1$.
    
Since $F_{1-i}^{(k)}=0$ for every $i \geq 1$,
\[
    \sum_{j=0}^{k-1}F_{n-1-j}^{(k)}\left(\sum_{i=0}^{k-1-j}F_{1-i}^{(k)} \right) = F_{n-1}^{(k)} + \dots +F_{n-k}^{(k)} = F_n^{k},
\]
which proves the case $\ell=n-1$. 
Next we suppose the lemma holds for some $\ell$, with $1 < \ell \leq n-1$. Hence, 
\begin{align*}
    F_n^{(k)} &= F^{(k)}_\ell \left(\sum_{i=0}^{k-1}F^{(k)}_{n-\ell-i} \right) + \sum_{j=1}^{k-1}F^{(k)}_{\ell-j} \left( \sum_{i=0}^{k-1-j}F^{(k)}_{n-\ell-i} \right) \\
    &= F^{(k)}_\ell \cdot F_{n-\ell+1}^{(k)} + \sum_{j=1}^{k-1}F^{(k)}_{\ell-j} \left( \sum_{i=0}^{k-1-j}F^{(k)}_{n-\ell-i} \right) \\
    %&= \left( F^{(k)}_{\ell-1}+ \dots +F^{(k)}_{\ell-k}    \right) \cdot F^{(k)}_{n-\ell+1} +  \sum_{j=1}^{k-1}F^{(k)}_{\ell-j} \left( \sum_{i=0}^{k-1-j}F^{(k)}_{n-\ell-i} \right) \\ 
    &=  \left( \sum_{j=1}^{k-1}F^{(k)}_{\ell-j} + F^{(k)}_{\ell-k}  \right) \cdot F^{(k)}_{n-\ell+1} +  \sum_{j=1}^{k-1}F^{(k)}_{\ell-j} \left( \sum_{i=0}^{k-1-j}F^{(k)}_{n-\ell-i} \right) \\
    %&= \sum_{j=1}^{k-1}F^{(k)}_{\ell-j}  \left( F^{(k)}_{n-\ell+1}  + \sum_{i=0}^{k-1-j}F^{(k)}_{n-\ell-i}    \right) + F^{(k)}_{\ell-k} \cdot F^{(k)}_{n-\ell+1} \\
    &=  \sum_{j=1}^{k-1}F^{(k)}_{\ell-j}  \left( F^{(k)}_{n-\ell+1}  + \sum_{i=1}^{k-j}F^{(k)}_{n-(\ell-1)-i}    \right) + F^{(k)}_{\ell-k} \cdot F^{(k)}_{n-\ell+1} \\
    &= \sum_{j=1}^{k-1}F^{(k)}_{\ell-j}  \left( \sum_{i=0}^{k-j}F^{(k)}_{n-(\ell-1)-i}    \right) + F^{(k)}_{\ell-k} \cdot F^{(k)}_{n-\ell+1}  \\
    &= \sum_{j=0}^{k-2}F^{(k)}_{(\ell-1)-j}  \left( \sum_{i=0}^{k-1-j}F^{(k)}_{n-(\ell-1)-i}    \right) + F^{(k)}_{(\ell-1)-(k-1)} \cdot F^{(k)}_{n-(\ell-1)},
\end{align*}
and at last,
\[
F_n^{(k)} = \sum_{j=0}^{k-1}F^{(k)}_{(\ell-1)-j}  \left( \sum_{i=0}^{k-1-j}F^{(k)}_{n-(\ell-1)-i} \right),
\]
which gives the desired identity.
\end{proof}

%%%%%%%%%%%%%%%%%%%%%%%%%%%%%%%%%%%%%%%%%%%%%%%%%
%%%%%%%%%%%%%%%%%%%%%%%%%%%%%%%%%%%%%%%%%%%%%%%%%
Next we deal with the Diophantine equation \eqref{eq-firstgoal} over the $k$-generalized Fibonacci sequence, generalizing \cite[Thm. 1.1]{chaves-marques-14}. In \cite{gomez-luca-14}, the authors consider the following theorem for $d=1$ and higher powers. 

\begin{thm} \label{thm-main1} Let $n\geq 1$ and $k \geq 3$. The Diophantine equation
\begin{equation} \label{kfibquad}
    (F_{n}^{(k)})^2 + (F_{n+d}^{(k)})^2 = F_{m}^{(k)}
\end{equation} 
has solution if, and only if, either $(n,d,m) = (1,0,3)$, or $(n,d,m) = (1,1,3)$ or $(n,d,m) = (a,0,2a-1)$ for $2 \leq a \leq \lfloor \frac{k+2}{2}\rfloor$. 
\end{thm}

\begin{proof}
First, let $d=0$. Thus \eqref{kfibquad} yields the following equation
\[
2(F_{n}^{(k)})^2 = F_{m}^{(k)}.
\]
If $n=1$, then  $(n,d,m)=(1,0,3)$ is a solution for \eqref{kfibquad}. If $2 \leq n \leq \lfloor \frac{k+2}{2} \rfloor$, then $3 \leq 2n-1 \leq k+1$. Thus, follows from Lemma \ref{kfibbasic} that $\F_n = 2^{n-2}$ and, for $m= 2n-1$, that 
 $\F_m = 2^{m-2}$. Hence, we have the solutions for \eqref{kfibquad} given by  
 $(n,d,m) =(a,0,2a-1)$, for $2 \leq a \leq \lfloor \frac{k+2}{2} \rfloor$. Next, suppose that $n \geq k+2$. By Lemma \ref{kfibbasic}, we can determine the possible values for $m$ in terms  of $n$. We first observe that
\[ 
   \alpha^{2n-3} < 2(F_{n}^{(k)})^2 = F_{m}^{(k)} \leq  \alpha^{m-1}, \ \]
   from where $ m \geq 2n-1$. Moreover, \\ 
 \[   \alpha^{2n} > 2(F_{n}^{(k)})^2 = F_{m}^{(k)} \geq  \alpha^{m-2}, \] 
 which implies $ m \leq 2n+1$. Therefore, $m \in \{2n-1, 2n, 2n+1\}$. Since $\F_{k+2} = 2^k-1$ is odd, it provides no solution. We now consider $m \geq k+3$, which implies $ n \geq (k+2)/2 > 2$. 
Replacing $n$ by $2n$ and letting $\ell = n$, Lemma ~\ref{lemkbona} yields the following inequality
\begin{align*}
\F_{2n} & =  \sum_{j=0}^{k-1}F_{n-j}^{(k)} \left( \sum_{i=0}^{k-1-j}F_{n-i}^{(k)} \right) \\
        & = F_{n}^{(k)} \left( \sum_{i=0}^{k-1}F_{n-i}^{(k)} \right) + \sum_{j=1}^{k-1}F_{n-j}^{(k)} F_{n}^{(k)} + \sum_{j=1}^{k-2}F_{n-j}^{(k)}  \left( \sum_{i=1}^{k-1-j}F_{n-i}^{(k)} \right) \\
        & \geq F_{n}^{(k)}F_{n+1}^{(k)} + F_{n}^{(k)}(F_{n}^{(k)} - F_{n-k}^{(k)}) + (\F_{n-1})^2 \\
        & > F_{n}^{(k)}(2F_{n}^{(k)} + F_{n-1}^{(k)} + \cdots + F_{n+1-k}^{(k)} - F_{n-k}^{(k)}) \\
        & \geq 2(\F_{n})^2. \
\end{align*}
Therefore $\F_{2n+1} \geq \F_{2n} > 2(\F_{n})^2$, and there are no solutions for $m=2n$ or $2n+1$. For $m=2n-1$, again Lemma ~\ref{lemkbona} for $\ell = n$ yields 

\begin{align*}
\F_{2n-1} & =  \sum_{j=0}^{k-1}F_{n-j}^{(k)} \left( \sum_{i=0}^{k-1-j}F_{n-1-i}^{(k)} \right) \\ 
        & = F_{n}^{(k)} \left( \sum_{i=0}^{k-1}F_{n-1-i}^{(k)} \right) + \sum_{j=1}^{k-1}F_{n-j}^{(k)} \left( \sum_{i=0}^{k-1-j}F_{n-1-i}^{(k)} \right) \\
        & < (F_{n}^{(k)})^2 +  \left(\sum_{j=1}^{k-1}F_{n-j}^{(k)}\right) \left( \sum_{i=0}^{k-2}F_{n-1-i}^{(k)} \right)  \\
        & \leq  2(F_{n}^{(k)})^2.
\end{align*}
Thus $\F_{2n-1} < 2(F_{n}^{(k)})^2$ and we also have no solutions in this case. 

Next, let $d \geq 1$. We observe that $(n,d)=(1,1)$ provides a solution to \eqref{kfibquad} for all $k \geq 3$. Furthermore,  when $(n,d)=(2,1)$ or $(1,2)$, \eqref{kfibquad} only has solution for $k=2$.  We are left with the case $n+d>3$. Replacing $n$ by $2n+2d-1$ and letting $\ell = n+d$ in the Lemma \ref{lemkbona}, 
\[
    F_{2n+2d-1}^{(k)}  = \F_{n+d}\left( \sum_{i=0}^{k-1}\F_{n+d-1-i} \right) + \sum_{j=1}^{k-1}F_{n+d-j}^{(k)} \left( \sum_{i=0}^{k-1-j}\F_{n+d-1-i} \right) > (F_{n+d}^{(k)})^2+(F_{n+d-1}^{(k)})^2. \]
Hence,    
    \begin{equation} \label{desigkfibquad1}
        F_{2n+2d-1}^{(k)} > (F_n^{(k)})^2+(F_{n+d}^{(k)})^2 .
    \end{equation}
Since 
\[\sum_{j=2}^{k-1}F_{n+d-j}^{(k)} \left( \sum_{i=0}^{k-1-j}F_{n+d-2-i}^{(k)} \right)<\left(\sum_{j=0}^{k-2}F_{m+d-2-j}^{(k)}\right)\left(\sum_{j=0}^{k-2}F_{m+d-2-j}^{(k)}\right),\]
adding $ F_{m+d-1}^{(k)} \left( \sum_{i=0}^{k-2}F_{n+d-2-i}^{(k)} \right)$ on both sides of the inequality, yields 
\[
\sum_{j=1}^{k-1}F_{n+d-j}^{(k)}\left(\sum_{i=0}^{k-1-j}F_{n+d-2-i}^{(k)}\right) < \left(\sum_{j=0}^{k-1}F_{n+d-1-j}^{(k)}\right)\left(\sum_{j=0}^{k-2}F_{n+d-2-j}^{(k)}\right) 
 = \F_{n+d}\left(\F_{n+d} - \F_{n+d-1}\right).
\]
Thus, replacing $n$ by $2n+2d-2$ and letting $\ell = n+d$ in the Lemma~\ref{lemkbona}, 
\begin{equation} \label{desigkfibquad2} F_{2n+2d-2}^{(k)} = F_{n+d}^{(k)}F_{n+d-1}^{(k)}+\sum_{j=1}^{k-1}F_{n+d-j} \left(\sum_{i=0}^{k-1-j}F_{n+d-2-i}^{(k)} \right)< (F_{n}^{(k)})^2 + (F_{n+d}^{(k)})^2 . 
\end{equation}
Combining \eqref{desigkfibquad1} and \eqref{desigkfibquad2} yields
\[ F_{2n+2d-2}^{(k)} < (F_{n}^{(k)})^2+(F_{n+d}^{(k)})^2 < F_{2n+2d-1}^{(k)}.
\]
Therefore, for $d \geq 1$ and $n+d>3$,  the sum of squares of two $k$-generalized Fibonacci numbers falls into two consecutive terms of the sequence, ensuring no solutions in this remaining case.   
\end{proof}

%%%%%%%%%%%%%%%%%%%%%%%%%%%%%%%%%%%%%%%%%%%%%%%%%%%%%%%%%%%%%%
%%%%%%%%%%%%%%%%%%%%%%%%%%%%%%%%%%%%%%%%%%%%%%%%%%%%%%%%%%%%%%
%%%%%%%%%%%%%%%%%%%%%%%%%%%%%%%%%%%%%%%%%%%%%%%%%%%%%%%%%%%%%%
%%%%%%%%%%%%%%%%%%%%%%%%%%%%%%%%%%%%%%%%%%%%%%%%%%%%%%%%%%%%%%
%%%%%%%%%%%%%%%%%%%%%%%%%%%%%%%%%%%%%%%%%%%%%%%%%%%%%%%%%%%%%%
%%%%%%%%%%%%%%%%%%%%%%%%%%%%%%%%%%%%%%%%%%%%%%%%%%%%%%%%%%%%%%

\section{Sum of two $s$-powers of Fibonacci numbers}

Let $s > 2$. The aim of this section is to solve completely the following Dipohantine equation
\begin{equation} \label{eq-spowers-sec3}
    F_n^s+F_{n+d}^s=F_m 
\end{equation}
This means to generalize Theorem \ref{propfibquad} for higher powers.  To achieve our goal, we need the following lemmas.

 \begin{lemma} Let $A : = \{(2,1)\}\cup\{(1,d) \ \big| \ d \in \mathbb{N}\}$. If $(s,d) \in \mathbb{N}^2 \setminus A$,  then
\[1+\alpha^{sd} - \alpha^{n} \sqrt{5}^{s-1} \neq 0\]
for every $n\in \mathbb{Z}$.
 \end{lemma}

\begin{proof} Let $(s,d) \in \mathbb{N}^2$ be such that
\begin{equation}\label{eq3.1-lemma3.1}
     1+\alpha^{sd}= \alpha^{n} \sqrt{5}^{s-1} 
\end{equation}
for some $n \in \mathbb{Z}$.

Let $\gamma := 1+\alpha^{sd} \in \Q(\sqrt{5})$ and consider the linear operator
     $m_{\gamma} :\mathbb{Q}(\sqrt{5}) \rightarrow  \mathbb{Q}(\sqrt{5})$ given by the multiplication for $\gamma$. 
The matrix of $m_{\gamma}$ in the basis $\{1, \alpha\}$ is given by 
\[\begin{pmatrix}
 F_{sd-1}+1 & F_{sd} \\
 F_{sd} & F_{sd+1}+1\\
 \end{pmatrix}.\]
Let $N_{\mathbb{Q}(\sqrt{5})}$ be the norm map relative to the field extension $\mathbb{Q}(\sqrt{5})/\Q$. The above discussion yields 
\[ N_{\mathbb{Q}(\sqrt{5})}(\gamma)= F_{sd-1}  F_{sd+1} + F_{sd-1} + F_{sd+1} +1 -F_{sd}^2 .\]
Thus, the Cassini identity $F_{n+1}F_{n-1}-F_n^2=(-1)^n$ implies
\[N_{\mathbb{Q}(\sqrt{5})}(\gamma)=(-1)^{sd}+F_{sd+1}+F_{sd-1}+1.\]
Since $N_{\mathbb{Q}(\sqrt{5})}(\alpha)=-1$ and $N_{\mathbb{Q}(\sqrt{5})}(\sqrt 5)=-5$, applying $N_{\mathbb{Q}(\sqrt{5})}$ to the identity (\ref{eq3.1-lemma3.1}) yields
 \[N_{\mathbb{Q}(\sqrt{5})}(\gamma)= (-1)^{sd}+F_{sd+1}+F_{sd-1}+1 = (-5)^{s-1}(-1)^n.\]
The left side of above identity is always positive for $sd \ge 2$, we thus conclude that $n$ is even and 
\begin{equation} \label{eq3.2-lemma3.1}
  N_{\mathbb{Q}(\sqrt{5})}(\gamma)=  (-1)^{sd}+F_{sd+1}+F_{sd-1}+1 = 5^{s-1}.
\end{equation}

Next we split the proof in three cases. We first suppose that $sd$ is odd. In this case, the identity (\ref{eq3.2-lemma3.1}) becomes 
 \[F_{s(d+1)}+F_{s(d-1)}=5^{s-1}.\]
Observe that $F_{n+1}+F_{n-1}$ stands for the $n$-th Lucas number. Hence  $5 \nmid F_{n+1}+F_{n-1}$ for all $n \in \mathbb{N}$ and, thus, $s=1$.

If $sd$ is even, we observe that the determinant of the matrix of the operator $m_{\gamma}$ equals its trace. Let $Tr_{\mathbb{Q}(\sqrt{5})}$ stand for the trace map relative to the field extension $\Q(\sqrt{5})/\Q$. If $s$ is odd, applying 
$Tr_{\mathbb{Q}(\sqrt{5})}$ to the identity (\ref{eq3.1-lemma3.1}) yields 
\[ Tr_{\mathbb{Q}(\sqrt{5})}(1+\alpha^{sd})=
\sqrt{5}^{s-1} Tr_{\mathbb{Q}(\sqrt{5})}(\alpha^n)=\sqrt{5}^{s-1}(F_{n+1}+F_{n-1})\]
Hence in the identity (\ref{eq3.2-lemma3.1}) we have 
\[ 5^{s-1}= N_{\mathbb{Q}(\sqrt{5})}(\gamma)=Tr_{\mathbb{Q}(\sqrt{5})}(\gamma)=\sqrt{5}^{s-1}(F_{n+1}+F_{n-1}).\]
which implies $F_{n+1}+F_{n-1}=\sqrt{5}^{s-1}.$ Then $s=1$.

We are  left to the case $sd$ and $s$ are even. Applying 
$Tr_{\mathbb{Q}(\sqrt{5})}$ to the identity (\ref{eq3.1-lemma3.1}) yields
\[ Tr_{\mathbb{Q}(\sqrt{5})}(1+\alpha^{sd})=\sqrt{5}^{s-2} Tr_{\mathbb{Q}(\sqrt{5})}(\sqrt{5}.\alpha^n)=\sqrt{5}^{s} F_n\]
Since $N_{\mathbb{Q}(\sqrt{5})}(\gamma)=Tr_{\mathbb{Q}(\sqrt{5})}(\gamma)$, 
in identity (\ref{eq3.2-lemma3.1}) we have $5^{s-1}=\sqrt{5}^{s}F_n$. Thus, $F_n =\sqrt{5}^{s-2}$.

By the Carmichael Theorem the only powers of $5$ in the Fibonacci sequence are $F_1=1$ and $F_5=5$. Hence either $s=2$ and $n=1$ or $s=4$ and $n=5$. But only the first case is a solution for $(\ref{eq3.1-lemma3.1})$, where we must have $d=1.$ 
\end{proof}

Although it is well-known that irrational numbers are, in a certain sense, well approximated by rational numbers, the following result explores the opposite direction: what is a ``safe'' distance of a given irrational number (namely, $\sqrt{5}$) from any rational number?

%%%%%%%%%%%%%%%%%%%%%%%%%%%%%%%%%%%%%%%%%%%%%%%%%%%%

 \begin{lemma} \label{lema3.2} Let $p,q \in \mathbb{Z}$ with $ p \geq 0$ and $q > 0$. Then
\begin{equation} \label{deslema3.2}
    \big|q\sqrt{5} - p\big| \ge \frac{1}{6q} .
\end{equation}
\end{lemma}

\begin{proof} We start by pointing out that,
    \[\big|q\sqrt{5} -p\big|=\frac{\big|5q^2-p^2\big|}{q\sqrt{5} +p} \geq \frac{1}{q\sqrt{5} +p} ,\]    
 where the last inequality holds, since $\sqrt{5}$ is not a rational number. Notice that for a fixed $q$, two possible values for $p$ minimizes the LHS of \eqref{deslema3.2}: $p=\lfloor q\sqrt{5} \rfloor$ or $\lceil q \sqrt{5}\rceil$. Since $\max \{ \lfloor q\sqrt{5} \rfloor , \lceil q \sqrt{5}\rceil \} \leq 3q$, and by the previous inequality we have
 \begin{align*}
\big|q\sqrt{5}-p\big|  & \ge \min \left\{ \big|q\sqrt{5}-\lfloor q\sqrt{5}         \rfloor \big|, \big|q\sqrt{5}-\lceil q\sqrt{5} \rceil\big| \right\} \\
                & \ge \min \left\{\frac{1}{q\sqrt{5}+\lfloor q\sqrt{5}\rfloor} , \frac{1}{q\sqrt{5}+\lceil q\sqrt{5}\rceil} \right\} \\
                & = \frac{1}{q\sqrt{5}+\lceil q\sqrt{5}\rceil} \ge \frac{1}{q(\sqrt{5}+3)}>\frac{1}{6q} \ .  
\end{align*}
Thus the lemma is proved.
\end{proof}

%%%%%%%%%%%%%%%%%%%%%%%%%%%%%%%%%%%%%%%%%%%%%%%%%%%%%%%%%%%%

The next lemma is also related to a lower bound for the distance between two given numbers, in this case a power of the golden ratio to a power of $\sqrt{5}$.

\begin{lemma} \label{lemma-alphabetalow} Let $s, n \in \mathbb{N}$, and $ s \ne 2,4$. Then, 
\[\big|\alpha^n-\sqrt{5}^{s-1}\big| \ge 1-\big|\beta\big|^n,\]
where $\alpha = (1+\sqrt{5})/2$ and $\beta = -\alpha^{-1}$.
\end{lemma}
\begin{proof} First, for an odd $s$, we have
\[\big|\alpha^n-(\sqrt{5})^{s-1}\big|=\big|(\alpha^n+\beta^n)-(\sqrt{5})^{s-1}-\beta^n\big|\ge \big|L_n-(\sqrt{5})^{s-1}\big|-\big|\beta\big|^n ,\]
and since $5 \nmid L_n$ for all $n \in \mathbb{N}$, then $\big|L_n-\sqrt{5}^{s-1}\big| \ge 1$, which yields
\[ \big|\alpha^n-(\sqrt{5})^{s-1}\big| \ge 1-\big|\beta\big|^n . \]
For the case where $s$ is even, an aditional algebraic manipulation yields
\begin{align*}
    \big|\alpha^x-(\sqrt{5})^{s-1}\big| & =\sqrt{5}  \;\left|\frac{\alpha^n-\beta^n+\beta^n}{\sqrt{5}}-(\sqrt{5})^{s-2}\right| \\
    & \ge \sqrt{5} \; \big|F_n-\sqrt{5}^{s-2}\big|-\big|\beta\big|^n.
\end{align*}
Moreover, since $s \neq 2,4$, and the only Fibonacci numbers that are powers of $5$ are $F_1=F_2=1$ and $F_5=5$, we obtain $\big|F_n-(\sqrt{5})^{s-2}\big|\ge 1$. Hence
\[\big|\alpha^n-(\sqrt{5})^{s-1}\big| \ge \sqrt{5}-\big|\beta\big|^n>1-\big|\beta\big|^n ,\]
and we are done.
\end{proof}

Let us consider the extension of the Fibonacci sequence for negative index. Namely, we define for $n \in \Z$,  $F_{-n} :=(-1)^{n+1} F_n$. We observe that the extended Fibonacci sequence still satisfy $F_{n+1}=F_n+F_{n-1}$, but now for every $n \in \mathbb{Z}$. The following lemma generalizes Lemma \ref{kfibbasic} item (iv).  

%%%%%%%%%%%%%%%%%%%%%%%%%%%%%%%%%%%%%%%%%%%%%%%%%%%%%%%%%%%%%%

\begin{lemma} \label{lema3.5} Let $(F_n)_{n \in \Z}$ be the extended Fibonacci sequence. Then the following identity
\[\alpha^n=\alpha F_n+F_{n-1}\]
hods for every $n \in \mathbb{Z}$. 
\end{lemma}

\begin{proof}
For $n \geq 0$, this is the Lemma \ref{kfibbasic} item (iv). For $n=0$,
$ 1 = \alpha^n = \alpha F_0+F_{-1}.$
For $n=-1$, $\alpha^{-1}=\alpha-1=\alpha F_{-1}+F_{-2}.$
Suppose the identity holds for $k>n$, then
\[\alpha^n=\alpha^{n+2}-\alpha^{n+1}=\alpha(F_{n+2}-F_{n+1})+F_{n+1}-F_n=\alpha \cdot F_n+F_{n-1},\]
which concludes the proof. 
\end{proof} 

%%%%%%%%%%%%%%%%%%%%%%%%%%%%%%%%%%%%%%%%%%%%%%%%%%%%%%%%%%%%%%%%%

The following three lemmas focus specifically on the variables in the Diophantine equation \eqref{eq-spowers-sec3}. These lemmas provide useful upper and lower bounds for \( n \), \( d \), \( s \), and \( m \), which will be central to proving the main result of this section.

\begin{lemma} \label{lemma3.6} Let $d,n,m,s \in \Z_{> 0}$ satisfying \eqref{eq-spowers-sec3}. Then
\[s(n+d-2)+1 < m \leq s(n+d-1)+2.\]
\end{lemma}

\begin{proof}
By the well-known inequality $\alpha^{k-2}<F_k <\alpha^{k-1},$
\[ \alpha^{m-1}>F_m=F_n^s+F_{n+d}^s> F_{n+d}^s > \alpha^{(n+d-2)s} \]
which implies that $m>s(n+d-2)+1.$
On the other hand, 
\begin{align*} 
\alpha^{m-2}< F_m=F_n^s+F_{n+d}^s & < \alpha^{(n-1)s}+\alpha^{(n+d-1)s} \\
& < \alpha^{ns-s+ds-1} +  \alpha^{ns+ds-s} \\
& = \alpha^{ns+ds-s} (1 + \alpha^{-1}) \\
& = \alpha^{ns+ds-s+1}.
\end{align*} 
Therefore, $ m \le s(n+d)-s+2.$
\end{proof}

%%%%%%%%%%%%%%%%%%%%%%%%%%%%%%%%%%%%%%%%%%%%%%%%%%%%%%%%%%%%%%%%%%%

In order to obtain another estimate for $n,d$ in terms of $s$, we need the following lemma.
 
%%%%%%%%%%%%%%%%%%%%%%%%%%%%%%%%%%%%%%%%%%%%%%%%%%%%%%%%%%%% 

 \begin{lemma} \label{lema3.4}  Let $n, s \in \mathbb{N}$, where $s \geq 1$. Then 
 \[ \big|(F_n\sqrt{5})^s -\alpha^{ns}\big|<2^s \; \alpha^{n(s-2)}.\] 
 Moreover, if $n>\log_{\alpha}(s)$,  
 \[ \big|(F_n\sqrt{5})^s -\alpha^{ns}\big|< 2s  \; \alpha^{n(s-2)} . \]
\end{lemma}
\begin{proof} We begin by showing a useful inequality for the proof of both statements. Recall that taking $\alpha=(1+\sqrt{5})/2$, and $\beta = -\alpha^{-1}$, Binet's formula combined with the binomial theorem, gives 
\begin{align*}
    \big|(F_n\sqrt{5})^s -\alpha^{ns}\big|
    & =\big|(\alpha^n-\beta^n)^s-\alpha^{ns}\big| \\
    & = \left|\sum_{k=0}^{s}\binom{s}{k} (-1)^{k}\alpha^{n(s-k)}  \beta^{kn}-\alpha^{ns}\right| \\
    & = \left|\sum_{k=1}^{s} \binom{s}{k} (-1)^{k(n+1)} \alpha^{n(s-2k)} \right| \\
    & \leq  \sum_{k=1}^{s} \binom{s}{k} \alpha^{n(s-2k)} \\
    & = \alpha^{n(s-2)}  \sum_{k=1}^{s} \binom{s}{k} \alpha^{-2(k-1)n} ,
\end{align*}
hence we have,
\begin{equation} \label{desigfibalpha}
    \big|(F_n\sqrt{5})^s -\alpha^{ns}\big| \leq \alpha^{n(s-2)}  \left[s+\sum_{k=2}^{s}\binom{s}{k} \alpha^{-kn}\right] .
\end{equation}
To obtain the first statement, simply notice that  $\alpha^{-kn} < 1$, for all $2 \leq k \leq s$, and then \eqref{desigfibalpha} consequently yields 
\[
\big|(F_n\sqrt{5})^s -\alpha^{ns}\big| < \alpha^{n(s-2)}  \sum_{k=1}^{s}\binom{s}{k} < 2^{s}\alpha^{n(s-2)} ,
\]
as desired. Now, if $n > \log_{\alpha}(s)$, then \eqref{desigfibalpha} provides
\[
\big|(F_n\sqrt{5})^s -\alpha^{ns}\big| \leq \alpha^{n(s-2)} \left[s+\sum_{k=2}^{s}\binom{s}{k}\alpha^{-kn}\right] < \alpha^{n(s-2)}[s+(\underbrace{1+\alpha^{-n}}_{<1+1/s})^s] < 2s \alpha^{n(s-2)},
\]
and we have the second inequality, which completes the prove.
\end{proof}

%%%%%%%%%%%%%%%%%%%%%%%%%%%%%%%%%%%%%%%%%%%%%%%%%%%%%%%%%%%%%%%%%

\begin{lemma} \label{lemma-3.7} Let $d,n,m,s \in \Z_{> 0}$, with $d \le 2$ and $s \ge 3$, satisfying \eqref{eq-spowers-sec3}. Then $n+d <3s.$ 
\end{lemma}

\begin{proof} Follows from the hypothesis and Lemma \ref{lema3.4} that 
\begin{align*}
\big|\alpha^m  \sqrt{5}^{s-1}-\alpha^{ns+ds}-\alpha^{ns}\big|
       & = \big|\sqrt{5}^{s-1}(\alpha^m-\sqrt{5}  F_m)-(\alpha^{ns+ds}-\sqrt{5}^s  F_{n+d}^s)-(\alpha^{ns}-\sqrt{5}^s  F_n^s)\big| \\
     & \leq \big|\beta\big|^m \sqrt{5}^{s-1}+2^s \alpha^{ns-2n}+2^s \alpha^{ns+ds-2(n+d)} \\
     &< 2^{s+1} \alpha^{s(n+d)-2(n+d)}+\big|\beta\big|^m \sqrt{5}^{s-1}.
 \end{align*}
Since $\big|\beta\big|<1$, Lemma \ref{lemma3.6} yields 
\begin{align*}
    \big|\beta\big|^m (\sqrt{5})^{s-1} <\big|\beta\big|^{s(n+d-2)+1}(\sqrt{5})^{s-1}   
    & =(\big|\beta\big|^2\sqrt{5})^{s-1}  \big|\beta\big|^{s(n+d-4)+3} \\ 
    & <\big|\beta\big|^{s(n+d-4)+3} \\
    & =\frac{\alpha^{4s-3}}{\alpha^{s(n+d)}} \\ 
    & < 2^{s+1}  \alpha^{s(n+d)-2(n+d)}.
\end{align*}
Hence, we obtain 
\begin{equation}\label{desig1lem3.7}
        \big|\alpha^m  \sqrt{5}^{s-1}-\alpha^{s(n+d)}-\alpha^{ns}\big| < 2^{s+2} \alpha^{s(n+d)-2(n+d)}.   
\end{equation}
We divide above inequality by $\alpha^{s(n+d)}$ to obtain
\begin{equation} \label{desig2lem3.7}
    \big|\alpha^{m-s(n+d)} \sqrt{5}^{s-1}  - \alpha^{-ds} - 1  \big|<2^{s+2} \alpha^{-2(n+d)}.
\end{equation}
From Lemma \ref{lema3.5}, the left-hand side of the above inequality can be written as
\begin{align*}
    \big|\alpha^{m-s(n+d)} \sqrt{5}^{s-1}  - \alpha^{-ds} - 1  \big|
& =\big| (\alpha F_{m-s(n+d)}+F_{m-s(n+d)-1}) \sqrt{5}^{s-1}  - (\alpha F_{-ds}+F_{-ds-1}) - 1 \big| \\
& =\frac{1}{2}  \big|F_{m-s(n+d)}\sqrt{5}^{s} +L_{m-s(n+d)}\sqrt{5}^{s-1}-F_{-sd}\sqrt{5}-L_{-sd}-2 \big| \\
& = \frac{1}{2} \big|A\sqrt{5}-B \big|, 
\end{align*}
for,
\[
A := F_{m-s(n+d)}\sqrt{5}^{s-1} +L_{m-s(n+d)}\sqrt{5}^{s-2}-F_{-sd} \ \ \mbox{ and } B:= L_{-sd}+2 \ .
\]
If $A=0$, then
\[
 \big|\alpha^{m-s(n+d)} \sqrt{5}^{s-1}  - \alpha^{-ds} - 1  \big| = \frac{1}{2}  \big|B \big| \geq \frac{1}{2} \ .
\]
If $A \neq 0$, Lemma \ref{lema3.2} yields
\[\frac{1}{2}  \big|A\sqrt{5}-B \big| > \frac{1}{12 \big|F_{-sd}-\delta(s)\big|} ,\]
where 
\[ \delta(s)= \begin{cases}
    \sqrt{5}^{s-1}  F_{m-ns-ds}  & \text{ if $s$ is odd} \\[2pt]
    \sqrt{5}^{s-2}L_{m-ns-ds} & \text{ if $s$ is even}. 
\end{cases} \]
In either cases, we have
\[\big|F_{-sd}-\delta(s)\big|\le \big|F_{-sd}\big|+\big|\delta(s)\big|<F_{2s}+(\sqrt{5})^{s-1}  L_{ns+ds-m}<\sqrt{5}^{s}  L_{2s-1}<\alpha^{4s}\]
Putting all together yields
$ (12  \alpha)^{-4s}< 2^{s+2} \alpha^{-2(n+d)}$. Thus
\[\alpha^{2(n+d)}<48\alpha^{4s}  2^s<\alpha^{8+6s}\]
from where we conclude that $ n+d<4+1.5s<3s$. 
\end{proof}

%%%%%%%%%%%%%%%%%%%%%%%%%%%%%%%%%%%%%%%%%%%%%%%%%%%%%%%%%%%%%

Next, we show a result similar to the previous lemma but for $d \ge 3$.

\begin{lemma}Let $d,n,m,s \in \Z_{> 0}$, with $d \ge 3$ and $s \ge 5$, satisfying \eqref{eq-spowers-sec3}. Then $n+d <3s .$
\end{lemma}

\begin{proof}
 By dividing both sides of inequality\eqref{desig1lem3.7} by $\alpha^m$, we obtain
\[ \big|\alpha^{ns+ds-m}+\alpha^{ns-m}-(\sqrt{5})^{s-1}\big|<2^{s+2} \; \alpha^{(s-2)(n+d)-m} .\]
On the other hand, the reverse triangle inequality, combined with Lemmas  \ref{lemma-alphabetalow}  and \ref{lemma3.6}, and the lower bounds for $d$ and $s$, yields
\begin{align*}
   \big|\alpha^{ns+ds-m}+\alpha^{ns-m}-\sqrt{5}^{s-1}\big| & \ge \big|\sqrt{5}^{s-1}-\alpha^{s(n+d)-m}\big|-\alpha^{ns-m} \\
            & > 1-\big|\beta\big|^{s(n+d)-m}-\alpha^{ns-m} \\
            & > 1-\big|\beta\big|^{2s-1}-\alpha^{-ds+2s-1} \\
            & >1-\alpha^{-9}-\alpha^{-5} > 1/2.
\end{align*}
By Lemma \ref{lemma3.6}, $s(n+d)-m < 2s-1$. Thus
\[
\frac{1}{2} < \big|\alpha^{ns+ds-m}+\alpha^{ns-m}-\sqrt{5}^{s-1}\big| < 2^{s+2} \;\alpha^{2s-1-2(n+d)}.
\]
Since $\alpha^2>2$,
\[\alpha^{2(n+d)}< 2^{s+3} \; \alpha^{2s-1}<\alpha^{4s+5}.\]
Therefore, $\ n+d \le 2s+2<3s$, which proves the lemma.
\end{proof}

%%%%%%%%%%%%%%%%%%%%%%%%%%%%%%%%%%%%%%%%%%%%%%%%%%%%%%%%%%%%%%%%%

In the following, we address a solution to \eqref{eq-spowers-sec3} in the cases $s=3$ and $s=4$.

\begin{proposition}
Let $d,n,m \in \Z_{> 0}$ and $s=3,4$ satisfying \eqref{eq-spowers-sec3}.
Then $n=d=1$ and $m=3$.
\end{proposition}

\begin{proof}
If $d \le 2$, Lemma \ref{lemma-3.7} yields $n+d<12$. Thus a straightforward computation yields $n=d=1$ and $m=3$.

If $d \ge 3$, inequality \eqref{desig2lem3.7} from Lemma \ref{lemma-3.7} yields
\[ \big|\alpha^{m-ns-ds}\sqrt{5}^{s-1}  - \alpha^{-sd} -1\big|<2^{s+2} \;\alpha^{-2(n+d)}.\]
Thus, 
\begin{align*}
    \big| 1- \alpha^{m-ns-ds} \sqrt{5}^{s-1}\big| - \alpha^{-3s}
    &< \big|1-\alpha^{m-ns-ds} \sqrt{5}^{s-1}\big| -\alpha^{-s d}\\ 
    & <2^{s+2} \; \alpha^{-2(n+d)} \\
    & < 2^6 \;\alpha^{-2(n+d)}.
\end{align*}
 
Let $t:= m-s(n+d)$, follows from Lemma \ref{lemma3.6} that 
$ 2-2s \le t \le 2-s.$
Moreover, for $s=3,4$, 
\[ \min_{t}\big\{\big|1- \alpha^{t} \sqrt{5}^{s-1} \big|-\alpha^{-3s} \big\}>0.005.\]
Therefore, $n+d<10$. A straightforward computation completes the proof.
\end{proof}

%%%%%%%%%%%%%%%%%%%%%%%%%%%%%%%%%%%%%%%%%%%%%%%%%%%%%%%%%%%%%%%%%%%%%%

 The following theorem solves the Diophantine equation \eqref{eq-spowers-sec3} and generalizes \cite[Thm. 1]{luca-oyono-11}.

\begin{thm} \label{thm-main2}  The Diophantine equation
\[ F_n^s+F_{n+d}^s=F_m \]
has no solution for $ (n,d) \not\in \{(1,2),(1,1)\}$. If $s=2$, $n=1$ and $d=2$ provide the unique nontrivial solution. Moreover, if $s \geq 3$, the unique solution  is $n = d = 1$.
\end{thm}

\begin{proof}
By the previous proposition, we may assume  $s \geq 5$. 
Rewriting the Diophantine equation in terms of the Binet formula yields
\[ \frac{\alpha^m}{\sqrt{5}}-F_{n+d}^s=\frac{\beta^m}{\sqrt{5}}+F_n^s.\]
Multiplying both sides by $F_{n+d}^{-s}$ yields
\begin{equation} \label{eq-1-thm3.11}
     \alpha^m \; (\sqrt{5})^{-1} \; F_{n+d}^{-s}-1 = 
\frac{\beta^m}{\sqrt{5}\; F_{n+d}^s }+\frac{F_n^s}{F_{n+d}^s}<\frac{2}{1,5^s},
\end{equation}
where the inequality follows the Lemma \ref{kfibbasic} $F_n/F_{n+d}<2/3$.

Next, we apply Lemma \ref{lemma-matveev} for $ \gamma_1=\alpha, \;  \gamma_2=\sqrt{5}, \;  \gamma_3=F_{n+d}, \;
 b_1 = m, \; b_2 = -1,  \text{ and }  b_3 = -s.$
We thus obtain  
\begin{equation} \label{eq-2-thm3.11}
    (m  e)^{- \lambda} < \big| \alpha^m \; (\sqrt{5})^{-1} \; F_{n+d}^{-s}-1\big|  
\end{equation}
where, in the notation of Lemma \ref{lemma-matveev}, $C_{3,2}=3.4 \cdot 10^{9}$,
$ A_1 = 1.61, A_2 = 0.5,  
A_3 = 0.5(n+d)$
and, therefore, $\lambda = 1.4 \cdot 10^{9}(n+d).$

From inequalities \eqref{eq-1-thm3.11} and \eqref{eq-2-thm3.11} we obtain
\[ -\lambda \log(me) < \log(2)-s\log(1.5). \]
Thus, 
\begin{align*}
    s & < \frac{\log(2)}{\log(1.5)}+ \frac{1.4}{\log(1.5)}  10^{9} \; (n+d) \log(m e) \\
    & <1,5 \cdot 10^9 \cdot (n+d) \cdot \log(m)\\
    & <1,5 \cdot 10^{9}(n+d) \cdot \log(s(n+d))
\end{align*}

If $s \geq n+d$, 
\[ s<1.5 \cdot 10^9 \; (n+d) \log((n+d)  s)\le 3 \cdot 10^9 \; (n+d)\log(s),\] 
which implies that 
\[ \frac{s}{\log(s)}<3 \cdot 10^9 \; (n+d)\] 
We next observe that if $a \in \R$ is such that $a>3$, then
$\frac{s}{\log(s)}<a$ implies that $s<2a  \log(a)$.
Thus
\begin{align*}
    s & < 3 \cdot 10^9 \; (n+d)  \log(3 \cdot 10^9(n+d)) \\
    & <3 \cdot 10^9 \; \log(3 \cdot 10^9)(n+d) \log(n+d) \\ 
 & < 2.22 \cdot 10^{12} \;  (n+d)  \log(n+d).
\end{align*}

In the following we split the proof in four cases. \\

\paragraph{Case 1.} 
Let $n+d<2 \log_{\alpha}(s)$. Hence, 
\[ s<2.22 \cdot 10^{12} \; (n+d)  \log(n+d) 
< 4.44 \cdot 10^{12} \; \log_{\alpha}(s) \log(2 \log_{\alpha}(s)).\]
A computational verification with \cite{Mathematica}, shows that 
$s<1.88 \cdot 10^{17}$ and $n+d<\log_{\alpha}(s)<165$. 

In what follows, we reduce the above upper bound for $s$. Let
\[ \Gamma : = m  \log(\alpha)-\log(\sqrt{5})-s \log(F_{n+d})\]
Since the left side of \eqref{eq-1-thm3.11} is positive,  $\Gamma > 0$.
Moreover
\[ \Gamma<e^{\Gamma}-1=\alpha^m \sqrt{5}^{-1}  F_{n+d}^{-s} - 1 < 2  (1.5)^{-s}.\] 
Hence, 
\[ m \log(\alpha)-\log(\sqrt{5})-s\log(F_{n+d})
< 2 (1.5)^s 
< 2  (1.5)^{\frac{m}{n+d-1}}
< 2  (1.5)^{\frac{m}{164}}.
\] 
Dividing above inequalities by $\log(F_{n+d})$ yields
\[m \frac{\log(\alpha)}{\log(F_{n+d})}-s-\frac{\log(\sqrt{5})}{\log(F_{n+d})}
< 2 (1.5^{\frac{1}{165}})^m\]
Next we apply Lemma \ref{lemma-dujella} for $A=2$, $B=1.5^{\frac{1}{165}}$, $\gamma_{n+d}= \log(\alpha)/\log(F_{n+d})$ and 
$\mu_{n+d}=- \log(\sqrt{5})/\log(F_{n+d})$. Let 
$q_{t}(\gamma_{n+d})$ be the denominator of the $t$-th convergence of  $\gamma_{n+d}$. 
We verify with computer assistance that $q_{99}(\gamma_{n+d}) > 11.28 \cdot 10^{17}$. Thus 
\[ \lVert \mu_{n+d} \;  q_{99}(\gamma_{n+d}) \rVert - 1.88 \cdot 10^{17} \lVert \gamma \; q_{99}(\gamma_{n+d})\rVert > 0 \]
for all $0 < n+d \le 164$. Therefore, 
\[ s< \max \Big\{\frac{\log(A \cdot q_{99}(\gamma_{n+d}) )}{\log(B)}\Big\}<58057.\] \\

\paragraph{Case 2.} 
Let $n+d>2 \log_{\alpha}(s)$ and $d \ge 4$. Follows from Lemma \ref{lema3.4} that
\begin{align*}
    \big|\sqrt{5}^{s-1}  \alpha^m-\alpha^{ns+ds}-\alpha^{ns}\big|
    & =\big|\sqrt{5}^{s-1}  (\alpha^m-\sqrt{5} \cdot F_m)-(\alpha^{ns+ds}-\sqrt{5}^s  F_{n+d}^s)-(\alpha^{ns}-\sqrt{5}^s F_n^s)\big| \\
    &\le \big|\sqrt{5}^{s-1} \alpha^m-\alpha^{ns+ds}-\alpha^{ns}\big|\\
    & \leq \big|\sqrt{5}^{s-1}(\alpha^m-\sqrt{5}F_m)\big|+\big|\alpha^{ns+ds}-\sqrt{5}^{s}  F_{n+d}^s\big|+\big|\alpha^{ns}-\sqrt{5}^s  F_n^s\big| \\
    &<\sqrt{5}^{s-1}  \big|\beta\big|^m+2s \;\alpha^{ns+ds-2n-2d}+2^s  \alpha^{ns-2n}.
\end{align*}
Since  $d \ge 4$, dividing above inequality by $\alpha^{n+d}$ yields
\begin{align*}
    \big|\alpha^{-sd}+1-\sqrt{5}^{s-1}  \alpha^{m-ns-ds}\big|
    & <2s \;  \alpha^{-2(n+d)}+2^s  \alpha^{-ds-2n}+\sqrt{5}^{s-1}  \big|\beta\big|^{m+ns+ds}\\
    &<4s \; \alpha^{-2(n+d)}+\big|\beta\big|^{m+ns+ds} \sqrt{5}^{s-1} \\
    &<5s \; \alpha^{-2(n+d)},
\end{align*}
where the second inequality follows from 
\[  \big|\beta\big|^m \sqrt{5}^{s-1} <\frac{\alpha^{4s-3}}{\alpha^{(n+d)  s}}\]
which results from Lemma \ref{lemma-alphabetalow}.  Hence, 
\[ \big|\alpha^{-sd}+1- \alpha^{m-ns-ds} \sqrt{5}^{s-1}\big|
< 5s \;  \alpha^{-2(n+d)}.\] 

The triangle inequality togheter the sum of $\alpha^{sd}$ yields
\[\big|1-\sqrt{5}^{s-1} \alpha^{m-ns-ds}\big|<5s \; \alpha^{-2(n+d)}+\alpha^{-sd}<5s \; \alpha^{-2(n+d)}+\alpha^{-4 \cdot \frac{n+d}{3}}<6s \; \alpha^{-(n+d)}.\]

Follows from Lemma \ref{lemma-matveev} that 
\[ (e \cdot 2s)^{-\lambda '} 
< \big(e \cdot \max\{s-1,|m-ns-ds|\}\big)^{-\lambda  }
< \big|1-\alpha^{m-ns-ds} \sqrt{5}^{s-1} \big|\]
where $\lambda$ is given as in the lemma and $\lambda ' : =2.8 \cdot 10^9$. Hence
\[ (2s \cdot e)^{-\lambda '} < 6s \; \alpha^{-2(n+d)}\] 
implies that
\[ \ \alpha^{2(n+d)}<6s(2s \cdot e)^{\lambda '}<s^{2\lambda '}.\]
Moreover, 
\[  n+d <2\lambda ' \; \log_{\alpha}(s)
< (5.6 \cdot 10^9)  \log\big( 2.22 \cdot 10^{12} (n+d)  \log(n+d) \big).\]
Therefore we obtain
\[  n+d< 6.99 \cdot 10^{17}\]
and that
\[  s <2.22 \cdot 10^{12}  (n+d)\log(n+d)<6.02 \cdot 10^{32}.\]

In what follows, we reduce the above upper bound for $s$. We first observe that 
\[ 
    \big|1-\alpha^{m-ns-ds} \sqrt{5}^{s-1} \big|< 
 \dfrac{5s}{\alpha^{2(n+d)}}+\dfrac{1}{\alpha^{sd}}
\]
Since $s<\alpha^n$, 
\[ 
\big|1-\alpha^{m-ns-ds}\sqrt{5}^{s-1} \big|  
< \dfrac{1}{\alpha^{ds}}+\dfrac{5}{\alpha^{n+2d}}.
\]
Moreover, since $s^2 < \alpha^n$,
\[ \big|1-\alpha^{m-ns-ds} \sqrt{5}^{s-1} \big|  < \dfrac{1}{\alpha^{ds}}+\dfrac{5}{\alpha^{2d}s^2}.
\]
Thus $| \Lambda | < 1/\alpha^5 + 1.91/25 < 0.17$, where 
 $\Lambda := 1-\alpha^{m-ns-ds} \sqrt{5}^{s-1}$ .
Let 
\[ \Gamma:=(s-1)\log(\sqrt{5})- \Big{(} (n+d) - m \Big{)} \log(\alpha).\] 
Since $e^{\Gamma}-1=\Lambda\neq0$, then $\Gamma\neq0$. Hence
$ |e^{\Gamma}-1|< 0.17$, which implies $ |\Gamma|< 0.19$ and that $e^{|\Gamma|}< 1.21.$
Thus 
\[  | \Gamma | < e^{|\Gamma|} \; |e^{\Gamma}-1| < 1.21 |\Lambda|
\quad \text{i.e., } \quad 
|\Gamma |< \dfrac{1.21}{\alpha^{sd}}+\dfrac{2.31}{s^2}.\]

Next we assume $s\geq174$. Then
$ | \Gamma | < 1/32(s-1)^2.$ Dividing by $(s-1) \cdot \log(\alpha)$ yields
\[ \left| \dfrac{\log(\sqrt{5})}{\log(\alpha)} - \dfrac{(n+d)s-m}{s-1} \right| < \dfrac{1}{32(s-1)^2}. \]

Let $ \log(\sqrt{5})/\log(\alpha) = [a_0 ; a_1 , a_2 , \ldots ]$ be its continued fraction representation.  
The Legendre criterion (Lemma \ref{lemma-legendre}) yields, 
\[\dfrac{p_t}{q_t} : = \dfrac{(n+d)s-m}{s-1}\]
is a convergent of the continued fraction of the irrational number $\log(\sqrt{5})/\log(\alpha)$. Let  $g :=  \gcd\big( (n+d)s-m, s-1 \big)$. Then 
\[ q_t=\dfrac{s-1}{g}<s < 6.02\cdot10^{32}.\]
We verify with computer assistance that $6.02 \cdot10^{32}< q_{70}$. Thus $t\in\{1,2,...,69\}$. Moreover, since
$\max\{a_1,a_2,...,a_{69}\}=29$, $a_t \leq 29.$
Therefore, 
\[  \left| \frac{\log(\sqrt{5})}{\log(\alpha)} - \dfrac{p_t}{q_t} \right| > \dfrac{1}{(a_t + 2)q_t^2}  .\]
Replacing $q_t$ by $(s-1)/g$ yields
\[ \left| \frac{\log(\sqrt{5})}{\log(\alpha)} - \dfrac{p_t}{q_t} \right| >  \dfrac{d^2}{31(s-1)^2}, \]
a contradiction. Therefore, $s \leq 173$. \\

\paragraph{Case 3.} Let $n<2\log_{\alpha}(s)$ and $d \le 3$. We first observe that 
\[ s< 2.22 \cdot 10^{12} \; (n+d)\log(n+d)
< 2.22 \cdot 10^{12} \;\big(2\log_{\alpha}(s)+3 \big) \big(\log(2\log_{\alpha}(s)+3)\big).\]
Hence, $ s < 1.91 \cdot 10^{17}$ and, therefore, 
\[  n+d<2\log_{\alpha}(s)+3 < 168. \]

In what follows, we reduce the above upper bound for $s$. Let $\Gamma=m \log(\alpha)-\log(\sqrt{5})-s\log(F_{n+d})$, as before,  \eqref{eq-1-thm3.11} it implies that  $\Gamma > 0$. Moreover, 
\[ \Gamma<e^{\Gamma}-1=\alpha^m \sqrt{5}^{-1} F_{n+d}^{-s}-1
< 2(1.5)^{-s}
\]
Thus 
\[ m \log(\alpha)-\log(\sqrt{5})-s\log(F_{n+d})
<2 (1.5)^s 
< 2(1.5)^{\frac{m}{n+d-1}}<2(1,5)^{\frac{m}{168}}
\]
Dividing by $\log(F_{n+d})$ yields
\[ m \frac{\log(\alpha)}{\log(F_{n+d})}-\frac{\log(\sqrt{5})}{\log(F_{n+d})} -s 
< 2 (1.5^{\frac{1}{168}})^m. \]

Next we apply Lemma \ref{lemma-dujella} for 
$A=2$, $B=1.5^{\frac{1}{165}}$, $\gamma_{n+d}=\frac{\log(\alpha)}{\log(F_{n+d})}$ and $\mu_{n+d}=-\log(\sqrt{5})/\log(F_{n+d})$.  
Let $q_t(\gamma_{n+d})$ be the denominator of the $t$-th convergente of $\gamma_{n+d}$. We verify with computer assistance that $q_{99}(\gamma_{n+d})>6.1 \cdot 10^{17}$ and 
\[\lVert \mu_{n+d} \; q_{99}(\gamma_{n+d}) \rVert - 1.91 \cdot 10^{17}  \lVert \gamma \cdot q_{99}(\gamma_{n+d}) \rVert >0\]
for all $0 < n+d \le 168$. Thus
\[s<\frac{\log(Aq / \varepsilon)}{\log(B)}<81766.\]

Let $M:= 1.91 \cdot 10^{17}$. For each $m$, let $q \in \Z$ be the denominator of the $99$-th convergent of  $\gamma$. We note that $q$ depends on $n+d$. The minimal value of $q$ for $n+d \in [3,168]$ exceed $10^{44} > 6M$. Hence, we might apply Lemma \ref{lemma-dujella} for each $q$, $\gamma$, and $\mu$. The maximum value of $M\lVert q \gamma\rVert$ is smaller than $10^{-27}$. While the minimum value of $\lVert q \mu \rVert$ is bigger than $2.4 \times 10^{-25}$. Thus, we might assume
\[ \varepsilon := \lVert q \mu \rVert - M \lVert q \gamma \lVert 
\; > 10^{-25}\]
in Lemma \ref{lemma-dujella}.  Moreover, 
the maximum value of  $q$ is smaller than $10^{61}$. Therefore, follows from above discussion and Lemma \ref{lemma-dujella} that all the solutions for the Diophantine equation satisfy
\[
 s < \frac{\log\left(Aq / \varepsilon\right)}{\log(B)} < \frac{\log\left(2 \times 10^{61}  \times 10^{25} \right)}{\log(1.5^{\frac{1}{165}})} < 81766.
 \]

\paragraph{Case 4.} 
Let $n>2 \log_{\alpha}(s)$ and $d \le 3$. Lemma \ref{lema3.4} yields 
\begin{align*}
 \big|\sqrt{5}^{s-1} \alpha^m-\alpha^{ns+ds}-\alpha^{ns}\big|
 & =\big|\sqrt{5}^{s-1} (\alpha^m-\sqrt{5}F_m)-(\alpha^{ns+ds}-\sqrt{5}^s F_{n+d}^s)-(\alpha^{ns}-\sqrt{5}^s F_n^s)\big|   \\
 & \le \big|\sqrt{5}^{s-1}(\alpha^m \sqrt{5}F_m)\big|+\big|\alpha^{ns+ds}-\sqrt{5}^{s} F_{n+d}^s\big|+\big|\alpha^{ns}-\sqrt{5}^s F_n^s\big| \\
& <\sqrt{5}^{s-1} \big|\beta\big|^m+2s \alpha^{ns+ds-2n-2d}+2s \alpha^{ns-2n}. 
\end{align*}
Dividing above inequality by $\alpha^{n+d}$ yields
\begin{align*}
    \big|\alpha^{-sd}+1-\sqrt{5}^{s-1}\big|& <2s\alpha^{-2(n+d)}+2s\alpha^{-ds-2n}+\sqrt{5}^{s-1} |\beta|^{m+ns+ds} \\
 & <4s \alpha^{-2(n+d)}+\sqrt{5}^{s-1}|\beta|^{m+ns+ds}\\ 
 & <5s \alpha^{-2(n+d)}
 \end{align*}
where the last inequality is obtained by applying the estimate 
\[ \sqrt{5}^{s-1}\big|\beta\big|^m<\frac{\alpha^{4s-3}}{\alpha^{(n+d)s}}\] 
arising from Lemma \ref{lemma-alphabetalow}. Thus 
\[ \big|\alpha^{-sd}+1-\sqrt{5}^{s-1} \alpha^{m-ns-ds}\big|<5s\alpha^{-2(n+d)}.
\]

Triangle inequality and adding $\alpha^{sd}$ yields 
\begin{equation} \label{eq-5.16}
    \big|1-\sqrt{5}^{s-1} \alpha^{m-ns-ds}\big|<5s \alpha^{-2(n+d)}+\alpha^{-sd}<6s \alpha^{-\frac{n+d}{3}}. 
\end{equation}

Follows from Lemma \ref{lemma-matveev} that
\[ (2se)^{-\lambda '} 
< (e \cdot \max\{s-1,|m-ns-ds\big|\})^{-\lambda '} 
< |1-\sqrt{5}^{s-1}\alpha^{m-ns-ds}\big|,\]
where $\lambda '=2.8\cdot 10^9$. Hence
\[ \alpha^{\frac{n+d}{3}}<6s(2se)^{\lambda '}<s^{2\lambda '},\]
and thus
\[ n+d <2\lambda ' \log_{\alpha}(s)
< 5.6 \cdot 10^9 \log\big(2.22 \cdot 10^{12} (n+d)\log(n+d) \big).\]
Therefore, $n+d<6,99.10^{17}$ and
\[ s<2.22 \cdot 10^{12} (n+d) \log(n+d)< 6.02 \cdot 10^{32}.\]

We next apply the above discussion to reduce the upper bound for $s$. Since $s<\alpha^n$, in the inequality \eqref{eq-5.16} we obtain 
\[ \big|1- \alpha^{m-ns-ds}\sqrt{5}^{s-1} \big|
<  \dfrac{1}{\alpha^{sd}} + \dfrac{5s}{\alpha^{2(n+d)}} 
< \dfrac{1}{\alpha^{sd}}+\dfrac{5}{\alpha^{n+2d}}
\] 
Moreover, since $\alpha^n>s^2$
\[ \big|1- \alpha^{m-ns-ds}\sqrt{5}^{s-1} \big|  
    < \dfrac{1}{\alpha^{ds}}+\dfrac{5}{\alpha^{2d}s^2}. \]
Thus $ | \Lambda | < \alpha^{-5} + 0.0764 < 0.17$, where 
$\Lambda := 1- \alpha^{m-ns-ds} \sqrt{5}^{s-1}$. 

Let $\Gamma:=(s-1)\log(\sqrt{5})- ( n+d - m ) \log(\alpha).$ Since $e^{\Gamma}-1=\Lambda\neq0$, then $\Gamma\neq0$ and 
 $|e^{\Gamma}-1|< 0.17$. Hence  
$|\Gamma|< 0.19$ and $e^{|\Gamma|}< 1.21$.  
Moreover, 
\[  | \Gamma | < e^{|\Gamma|} \; |e^{\Gamma}-1|< 1.21 \cdot |\Lambda|
\quad \text{ i.e., } \quad
 |\Gamma |< \dfrac{1.21}{\alpha^{sd}}+\dfrac{2.31}{s^2}.\] 

If $s\geq73$, then 
\begin{equation} \label{eq-5.17}
    | \Gamma | < \dfrac{1}{32(s-1)^2}
\end{equation}
Dividing by $\log(\alpha)(s-1)$ yields 
\[ \left| \dfrac{\log(\sqrt{5})}{\log(\alpha)} - \dfrac{(n+d)s-m}{s-1} \right| < \dfrac{1}{32(s-1)^2}.\] 

Let $ \log(\sqrt{5})/\log(\alpha) = [a_0 ; a_1 , a_2 , \ldots ]$ be its continued fraction representation.  
By Legendre's criterion (Lemma \ref{lemma-legendre}), 
\[ \frac{p_t}{q_t}:= \frac{((n+d)s-m)}{(s-1) }  \] is a convergent of the continued fraction of the irrational $\log(\sqrt{5})/\log(\alpha)$. 

Let $g:= \gcd \big( (n+d)s-m, s-1 \big)$. Then 
\[ q_t=\dfrac{s-1}{g}<s<6.02\cdot10^{32}.\] 
We verify with computer assistance that 
$6.02 \cdot 10^{32}< q_{70}$. Thus $t\in\{1,2,...,69\}$. 
Moreover, $\max\{a_1,a_2,...,a_{69}\}=29$, then $a_t\leq29$ for all $t = 1, \ldots, 69$. Hence, 
\[   \left| \frac{\log(\sqrt{5})}{\log(\alpha)} - \dfrac{p_t}{q_t} \right|
> \dfrac{1}{q_t^2 (a_t + 2)} .\] 
Replacing $q_t$ by $(s-1)/g$ yields
\[ \left| \frac{\log(\sqrt{5})}{\log(\alpha)} - \dfrac{p_t}{q_t} \right| > \dfrac{g^2}{31(s-1)^2}, \]
a contradiction with  inequality \eqref{eq-5.17}. Therefore, $s\leq72$. \\

It follows from the above cases that all variables of the Diophantine equation were bounded. Therefore, we conclude the proof through a computational verification of the remaining finite cases. \end{proof}

In the Appendix, we explain the computational method used in the proof of the previous theorem.

%%%%%%%%%%%%%%%%%%%%%%%%%%%%%%%%%%%%%%%%%%%%%%%%%%%%%%%%%%%%%%
%%%%%%%%%%%%%%%%%%%%%%%%%%%%%%%%%%%%%%%%%%%%%%%%%%%%%%%%%%%%%%
%%%%%%%%%%%%%%%%%%%%%%%%%%%%%%%%%%%%%%%%%%%%%%%%%%%%%%%%%%%%%%
%%%%%%%%%%%%%%%%%%%%%%%%%%%%%%%%%%%%%%%%%%%%%%%%%%%%%%%%%%%%%%
%%%%%%%%%%%%%%%%%%%%%%%%%%%%%%%%%%%%%%%%%%%%%%%%%%%%%%%%%%%%%%
%%%%%%%%%%%%%%%%%%%%%%%%%%%%%%%%%%%%%%%%%%%%%%%%%%%%%%%%%%%%%%

\section{Sums of $s$-powers of Fibonacci numbers} 

Theorem \ref{thm-main2} ensures that, except for the trivial cases, the sum of powers of two arbitrary Fibonacci numbers does not return to the sequence. In this section, we investigate whether we could sum an arbitrary number of powers of consecutive Fibonacci numbers and obtain a new Fibonacci number.

\begin{thm} \label{thm-main3}
    If $d+1<n$, then the Diophantine equation
\[ F_n^s+ F_{n+1}^s + \cdots +F_{n+d}^s = F_m
\]
has no solution for $n,s \geq 3$ and $d \geq 2$. 
\end{thm}

\begin{proof} Our proof consists of providing an upper bound for $s$ in terms of $m$, $n$, and $d$. We then use Matveev's Lemma \ref{lemma-matveev} to obtain an upper bound for $s$ that depends only on $n$ and $d$. Next, we use the hypothesis that $d+1 < n$ to obtain a numerical bound for $s$. After refining this bound, we show that for sufficiently large $n$, the Diophantine equation has no solution and verify computationally the remaining finite cases. \\

%%%%%%%%%%%%%%%%%%%%%%%%%%%%%%%%%%%%%%%%%%%%%%%%%%%%%%%%
%%%%%%%%%%%%%%%%%%%%%%%%%%%%%%%%%%%%%%%%%%%%%%%%%%%%%%%%

\paragraph{\textit{An upper bound for $s$ in terms of $n$, $m$ and $d$.}} 
Follows from Lemma \ref{kfibbasic} that
\[ \alpha^{m-2}  <  F_m  =  F_n^s + F_{n+1}^s + \cdots + F_{n+d}^s  
    <  \alpha^{(n-1)s}  +  \alpha^{ns}  + ... +  \alpha^{(n+d-1)s}  
    =  \alpha^{(n-1)s} \dfrac{\alpha^{(d+1)s}-1}{\alpha^s-1}.
 \]
Hence $\alpha^{m-2} < \alpha^{s(n+d-1)+1}$, which is equivalent to 
$ m-2  <  (n+d-1)s+1$. We thus conclude that 
\[ m < (n+d-1)s + 3  <  (n+d)s. \]

On the other side, Lemma \ref{kfibbasic}  yields 
\[ \alpha^{(n-2)s}   +   \alpha^{(n-3)s}   +...+   \alpha^{(n+d-2)s} < F_n^s + F_{n+1}^s + \cdots +  F_{n+d}^s  =  F_m < \alpha^{m-1}. 
\] 
Thereby, 
\[ \alpha^{(n-2)s} \; \alpha^{s(d+1)} \; \alpha^{-s+1} 
< \alpha^{m-1}\] 
and, thus $(n+d-2)s+1  <  m - 1$. Hence,  $(n+d-2)s+2  <  m$,
which implies that 
\[ (n+d-3)s <  m.\] 
Follows from above discussion that 
 \begin{equation}
     \dfrac{m}{n+d} < s < \dfrac{m}{n+d-3}
 \end{equation} 
i.e., a lower and upper bound for $s$ in terms of $m,n$ and $d$. \\

%%%%%%%%%%%%%%%%%%%%%%%%%%%%%%%%%%%%%%%%%%%%%%%%%%%%%%%%
%%%%%%%%%%%%%%%%%%%%%%%%%%%%%%%%%%%%%%%%%%%%%%%%%%%%%%%%

\paragraph{\textit{An upper bound for $s$ in terms of $n$ and $d$.}} 
We first model our Diophantine equation in order to apply the Matveev Lemma \ref{lemma-matveev}. The Binet formula yields
\[ 
\dfrac{\alpha^m}{\sqrt{5}} - F_{n+d}^s  =   F_n^s + F_{n+1}^s + \cdots + F_{n+d-1}^s +\dfrac{\beta^m}{\sqrt{5}}.
\]
Dividing by booth sides of above identity by $F_{n+d}^s$ yields
\[ 
\alpha^m \; \sqrt{5}^{-1} \; F_{n+d}^{-s} - 1  =  
\left( \dfrac{F_n}{F_{n+d}} \right)^s  
+  \left(\dfrac{F_{n+1}}{F_{n+d}} \right)^s  + \cdots 
+\left(\dfrac{F_{n+d-1}}{F_{n+d}} \right)^s  
+ \left(\dfrac{\beta^m }{ \sqrt{5}\; F_{n+d}} \right)^s. 
\] 
Follows from Lemma \ref{kfibbasic} that
\[ \alpha^m \; \sqrt{5}^{-1} \; F_{n+d}^{-s} - 1     
<   \left(\frac{2}{3}\right)^s  \; \dfrac{  (2/3)^{s(d+1)} -1  }{(2/3)^s-1} <  0.9^s .
 \]
From Lemma \ref{kfibbasic}, $\alpha^n= \alpha F_n + F_{n-1}$. Thus, if the left side of the above inequality vanishes, 
$ (\alpha F_m  + F_{m-1} ) \sqrt{5}^{-1}    =  F_{n+d}^s.$
Hence, it would imply 
\[  \dfrac{ F_m + 2F_{m-1} }{ \sqrt{5} } = 2F_{n+d}^s - F_m. \]
and, therefore, 
\[ \dfrac{ F_m + 2F_{m-1} }{ F_{n+d}^s - F_m } = \sqrt{5},
\]
a contradiction.  We thus might apply the Matveev Lemma \ref{lemma-matveev} for $\alpha^m \; \sqrt{5}^{-1} \; F_{n+d}^{-s} - 1$.  The constants are: $\ell =2$, $\gamma_1=\alpha$, $\gamma_2 = \sqrt{5}$, $\gamma_3 = F_{n+d}$, $B=m$, $A_1=0.5$, $A_2=1.61$ and $A_3 = 2(n+d-1)\log(\alpha) > 2\log(F_{n+d}) = h( F_{n+d} )$. Hence
\[ \exp(-C_{3,2}(1+ \log(B))A_1 A_2 A_3) 
< \alpha^m \; \sqrt{5} \; F_{n+d}^{-s} - 1 
< 0.9^s.
\]
Since $C_{3,2}<10^{12}$, 
\[ 0.9^{-s} < \exp(1.61 \cdot 10^{12}(1+\log(m))(n+d-1)\log(\alpha) ).
\]
We thus conclude that 
\begin{equation} \label{eq-s1}
    s < 4.72\cdot10^{13}(n+d+1)\log(n+d+1).
\end{equation} \\

%%%%%%%%%%%%%%%%%%%%%%%%%%%%%%%%%%%%%%%%%%%%%%%%%%%%%%%%
%%%%%%%%%%%%%%%%%%%%%%%%%%%%%%%%%%%%%%%%%%%%%%%%%%%%%%%%

\paragraph{\textit{An absolute bound for $s$.}} 
From now on, we assume $n\geq151$. By hypothesis, $d+1<n$, thus inequality \eqref{eq-s1} yields 
\begin{equation} \label{eq-5.11}
    s< 9.44\cdot 10^{13} \; n \log(2 n).
\end{equation}
Let $x_i :=  s /\alpha^{2(n+i)}$ for $i = 0, 1, \ldots, d$. The following inequality 
\begin{equation}  \label{eq-5.12}
    x_i < \dfrac{ 1 }{ \alpha^{n+i}}
\end{equation}
holds for $n \geq 80$ and every $i = 0, 1, \ldots, d$. 
Since $n\geq151$,  
\[ 1   <   \left(  1 + \dfrac{ 1 }{ \alpha^{2(n+i)}} \right)^s <  e^{x_i}   <  1 + 2x_i\]
and
\[  1 - 2x_i   <  e^{-2x_i}   <   \left( 1 - \dfrac{ 1 }{ \alpha^{2(n+i)} } \right)^s < 1
\]
also hold. Therefore,  
\begin{equation} \label{eq-5.13} 
    1 - 2x_i  <  \left( 1 - \frac{  (-1)^{n+i}  }{ \alpha^{2(n+i)} } \right)^s < 1 + 2x_i .
\end{equation}

By the Binet formula we obtain that
\[ F_{n+i}^s = \frac{ \alpha^{(n+i)s} }{ \sqrt{5} }  
\left( 1 - \frac{  (-1)^{n+i} }{ \alpha^{2(n+i)} } \right)^s.
\]
Hence, 
\[  F_{n+i} - \dfrac{ \alpha^{(n+i)s} }{  \sqrt{5}^s  }   
=  \dfrac{ \alpha^{(n+i)s} }{ \sqrt{5} }  
\left( \left( 1 - \frac{  (-1)^{n+i} }{ \alpha^{2(n+i)} }\right)^s - 1 \right). 
\]
The inequality \eqref{eq-5.13} yields
\begin{equation} \label{eq-5.6} 
    \left| F_{n+i}^s - \dfrac{ \alpha^{(n+i)s} }{ \sqrt{5}^s } \right| 
    < 2x_i \; \frac{ \alpha^{(n+i)s} }{ \sqrt{5}^s }.
\end{equation}

Next we return our attention to the Diophantine equation and write
\[ \dfrac{\alpha^m}{\sqrt{5}}  -  \dfrac{  \alpha^{ns}  }{ \sqrt{5}^s } 
- \dfrac{  \alpha^{(n+1)s}  }{ \sqrt{5}^s } 
- \cdots - \dfrac{ \alpha^{(n+d)s} }{ \sqrt{5}^s } 
=   \dfrac{\beta^m}{\sqrt{5}}  
+   F_n^s - \dfrac{  \alpha^{ns}  }{ \sqrt{5}^s }   
+   F_{n+1}^s - \dfrac{  \alpha^{(n+1)s  }}{ \sqrt{5}^s }    
+ \cdots +  F_{n+d}^s - \dfrac{  \alpha^{(n+d)s}  }{ \sqrt{5}^s }.
\]
Taking the modulus on both sides and applying the triangle inequality,
\[ 
 \left| \dfrac{  \alpha^m  }{  \sqrt{5}  }  -  \dfrac{  \alpha^{ns}  }{ \sqrt{5}^s } -  \cdots - \dfrac{ \alpha^{(n+d)s} }{ \sqrt{5}^s } \right| \leq  \dfrac{\beta^m}{\sqrt{5}}  +   \left| F_n^s - \dfrac{  \alpha^{ns}  }{ \sqrt{5}^s } \right|  +   \cdots +  \left| F_{n+d}^s - \dfrac{  \alpha^{(n+d)s}  }{ \sqrt{5}^s } \right|  . 
\]
Now, inequality \eqref{eq-5.6} ensures that,
\[
    \left| \dfrac{  \alpha^m  }{  \sqrt{5}  }  -  \dfrac{  \alpha^{ns}  }{ \sqrt{5}^s } -  \cdots - \dfrac{ \alpha^{(n+d)s} }{ \sqrt{5}^s } \right|  <  \dfrac{  2s\alpha^{ns-2n}  }{  \sqrt{5}^s  }  \left( 1 + \alpha^{s-2}  + \alpha^{2(s-2)}  + \cdots +  \alpha^{d(s-2)}  \right)  +  \dfrac{1}{\alpha^m} . 
 \]
Dividing both sides by $\alpha^{(m+d)s}/ \sqrt{5}^s$ and applying inequality \eqref{eq-5.12}, we obtain 
\begin{align*}
  \Big| \frac{\alpha^{ m}}{\alpha^{(m+d)s }} \sqrt{5}^{s-1} - \alpha^{-ds} - \alpha^{s(1-d)}   - \cdots  - 1 \Big| 
  & <  \dfrac{2x_0}{ \alpha^{ds}}  \left( 1 + \alpha^{s-2} +  \cdots
 +  \alpha^{d(s-2)} \right) + \dfrac{ \sqrt{5}^s }{ \alpha^{m+(m+d)s}} \\
 & <  \dfrac{2}{\alpha^{n+sd}} \left( \dfrac{  \alpha^{(s-2)(d+1)}  - 1 }{  \alpha^{s-2} - 1  } \right) +  \dfrac{ \sqrt{5}^s }{ \alpha^{m+(n+d)s}} .
\end{align*}
Let $\Lambda := \alpha^{ m - (n+d)s } \sqrt{5}^{s-1} - 1$. Follows from previous inequality that
\[    \left| \Lambda \right| 
<  \dfrac{2}{\alpha^{n+sd}}  \left( \dfrac{  \alpha^{(s-2)(d+1)}  - 1 }{  \alpha^{s-2} - 1  } \right)  + \dfrac{ \sqrt{5}^s }{ \alpha^{s(2n+2d-3)}}  
+  \alpha^{-s}  +  \alpha^{-2s}  + \cdots +   \alpha^{-ds}.
\]
Upper bounding the geometric progression yields
\begin{align*}
| \Lambda |  
& <  2\alpha^{ -n+ds} 
  \left( \dfrac{  \alpha^{(s-2)(d+1)}  - 1 }{  \alpha^{s-2} - 1  } \right)  + \dfrac{ \sqrt{5}^s }{ \alpha^{(2n+2d-3)s}}  +  \dfrac{1}{\alpha^s -1} \\
&  <    \dfrac{ 2\alpha^{s-2(d+1)} }{\alpha^{n} (\alpha^{s-2} -1)}  + \dfrac{ \sqrt{5}^s }{ \alpha^{(2n+2d -3)s} }  +  \dfrac{ 1 }{ \alpha^s -1 } \\ 
& < \dfrac{  2\alpha^s  }{ \alpha^n( \alpha^{s+4}  -  \alpha^{6}) }  +  \dfrac{ 1 }{  \alpha^{s-1}  }  +  \dfrac{ 1 }{ 2\alpha^{ s( n + 2d ) - 3s} }.
\end{align*}
Since $ (\sqrt{5}/\alpha^n)^s  <  1/2$ 
and $\alpha^s/ (\alpha^{s+4}  -  \alpha^6) < 0.42 $, 
\begin{equation} \label{Lambda1}
    \left| \Lambda \right|  
    <  \dfrac{ 0.84 }{ \alpha^n }  +  \dfrac{ 1 }{ \alpha^{s-1}} 
    +  \dfrac{ 1 }{ 2\alpha^{s(n+2d) - 3s}}.
\end{equation}
Moreover, if $b := \min\{ n , s - 1 \}$, then 
\begin{equation} \label{eq-5.15} 
    | \Lambda | < 2.34/\alpha^b.
\end{equation}
We observe that $\Lambda \neq 0$, since otherwise 
$\alpha^{2(n+d)s-2m} \in \Z$, which would imply that $s=1$, a contradiction.

We thus might apply the Matveed Lemma \ref{lemma-matveev} once more. The constantes now are:  $\ell =2$, $\gamma_1 =\alpha$, $\gamma_2 =\sqrt{5}$, $A_1 = 0.5$, $A_2 =1.61$ and $B = s-1$. Settling the Matveed Lemma \ref{lemma-matveev} with inequality \eqref{eq-5.15} yields
\[ \exp( -0.805 \cdot C_{2,2}  (1 + \log(s-1) )  ) < 2.34/\alpha^b,  
\]       
which implies that        
\[ b <  \dfrac{ \log(2.34) }{ \log(\alpha) } 
+ \dfrac{  4.27 \cdot 10^9  \; (1 + \log(s-1) )  }{ \log(\alpha) }.
\]
Thereby, $ b < 1.77  + 1.34\cdot10^{10} \;\log(s-1) < 1.35\cdot 10^{10} \; \log(s-1).
$

If $b=s-1$, $  (s-1)/ \log( s-1 ) < 1.35\cdot 10^{10}$
and, thus, $s< 3.6\cdot 10^{11}$. If $b=n$, inequality \eqref{eq-5.11} yields
\[ n < 1.35 \cdot 10^{10} \; \log \big( 9.44\cdot10^{13} \; n\log(2n) \big),
 \]
and, thus,  $ n < 8.51 \cdot 10^{11}$.

We thus apply \eqref{eq-5.11} to conclude that $ s<2.21\cdot10^{27}. $ In the following we improve this upper bound for $s$. \\

%%%%%%%%%%%%%%%%%%%%%%%%%%%%%%%%%%%%%%%%%%%%%%%%%%%%%%%%
%%%%%%%%%%%%%%%%%%%%%%%%%%%%%%%%%%%%%%%%%%%%%%%%%%%%%%%%

\paragraph{\textit{Refining the absolute bound for $s$.}}

Since \(m \geq 151\), \(s \geq 3\), and \(d \geq 2\), inequality \eqref{Lambda1} implies that 
\(\left| \Lambda \right| < 0.4\).

Let \(\Gamma := (s-1)\log(\sqrt{5}) - ((n+d)s - m)\log(\alpha)\). Since \(\Lambda = e^\Gamma - 1 \neq 0\), it follows that \(\Gamma \neq 0\). 
Since, $|e^{\Lambda} - 1| < 0.4$, we obtain that
$ |\Gamma | < 0.52$. Thus $e^{|\Gamma|} < 1.69.$
Thereby,
\[ |\Gamma| < e^{|\Gamma|}\big|e^\Gamma - 1\big| < 1.69 \cdot |\Lambda|. \] 
We thus establish the following upper bound for \(|\Gamma|\) in terms of \(\alpha\), \(s\), \(n\), and \(d\) 
\[ |\Gamma| < \frac{1.42}{\alpha^n} + \frac{1.69}{\alpha^{s-1}} + \frac{0.85}{\alpha^{(n+2d)s - 3s}}.
\]  
Dividing through by \((s-1)\log(\alpha)\) yields   
\[
\left| \frac{\log(\sqrt{5})}{\log(\alpha)} - \frac{(n+d)s - m}{s-1} \right| < \frac{1}{(s-1)\log(\alpha)}\left(\frac{1.42}{\alpha^n} + \frac{0.85}{\alpha^{s(n+2d) - 3s}} + \frac{1.69}{\alpha^{s-1}}\right).
\]  

If \(s \geq 19\), then  $ 154(s-1) < \alpha^{s-1}$ and  
\[ 
 154(s-1) < \alpha^{151} \leq \alpha^n,
\]  
where the latter follows from \(s < 2.21 \cdot 10^{27}\). Combining the previous inequalities, we obtain the subsequent approximation for $\log(\sqrt{5})/\log(\alpha)$ depending on $m,d,s,$ and $m$,
\begin{equation} \label{contfrac1}
     \left| \dfrac{ \log(\sqrt{5}) }{ \log(\alpha) } - \dfrac{  (n+d)s-m  }{ s-1 }  \right| < \dfrac{1}{(s-1)\log(\alpha)} \cdot \dfrac{4.8}{ 154(s-1)}  <  \dfrac{1}{32(s-1)^2}.
\end{equation}
Let $ \log(\sqrt{5})/\log(\alpha) = [a_0 ; a_1 , a_2 , \ldots ]$ be its continued fraction representation. The Legendre criterion (Lemma \ref{lemma-legendre}) and the previous inequality yield
\[ \frac{p_t}{q_t} := \frac{(n+d)s-m}{s-1}\] 
is a convergent of the continued fraction of $\log(\sqrt{5})/\log(\alpha)$. Let $g := \gcd( (n+d)s - m , s - 1)$. Then $q_t = (s - 1)/g $ and
$ \ q_t <s$. 
A quick computational check, gives that $q_{54} > 2.21\cdot 10^{27}$, and thus $t \in \{0,1,...,53\}$. Moreover, $a_k\leq29$ for all $k\in\{0,1,...,53\}$. Follows from a well-know property satisfied by continued fractions and their convergents that
\[ \left| \frac{\log(\sqrt{5})}{\log(\alpha)}- \dfrac{p_t}{q_t} \right| > \dfrac{1}{(a_{t+1} + 2) q_t^2}  >  \dfrac{g^2}{31(s-1)^2}  >  \dfrac{ 1 }{ 31(s-1)^2 },
\]
which contradicts \eqref{contfrac1}. Therefore, $s<19$. \\

%%%%%%%%%%%%%%%%%%%%%%%%%%%%%%%%%%%%%%%%%%%%%%%%%%%%%%%%
%%%%%%%%%%%%%%%%%%%%%%%%%%%%%%%%%%%%%%%%%%%%%%%%%%%%%%%%

%AQUI 
\paragraph{\textit{The final steps.}}

First, we consider $9 \leq s <19$. Inequality \eqref{Lambda1} yields
\[
    \left| \alpha^{ m - (n+d)s } 5^{ \frac{s-1}{2} } -   1 \right|  <  \dfrac{ 0.84 }{ \alpha^n }  +  \dfrac{ 1 }{ \alpha^{s-1}}  +  \dfrac{ 1 }{ 2\alpha^{s(n+2d) - 3s}}  <  \dfrac{ 1.84 }{ \alpha^n }  +  \dfrac{ 1 }{ \alpha^{s-1}}.
\]
Let $r:=m-(n+d)s$, and
\[
    \Theta(r,s) :=  \left| \alpha^{ m - (n+d)s } 5^{ \frac{s-1}{2} } -   1 \right| - \dfrac{1}{\alpha^{s-1}}. 
\]
The previous inequality implies that $\Theta(r,s) < 1.84/\alpha^n$. Next, since 
\[0.6-1.67s < r < 1.75 - 1.68s,\] 
for $9 \leq s < 19$, a computational search gives the following lower and upper bound
\[
     \dfrac{2.7}{10^3}  < \Theta(r,s) < \dfrac{1.8}{\alpha^n}.
\]
Thereby $n<13$, a contradiction with the fact that $n \geq 151$. Thus, we must have $3 \leq s\leq8$. 

To address the remaining cases, we first assert that the following inequality holds
\[\left|F_x^s \sqrt{5}^{s}-\alpha^{xs}\right|<2^s\;\alpha^{x(s-2)},\]
 for every $x \in \mathbb{N}$.
Thus, 
\begin{align*}
    \left|\sum_{k=n}^{n+d}\alpha^{ks}-\alpha^{m}  \; 5^{(s-1)/2}\right|
    & =  \left|\sum_{k=n}^{n+d}(\alpha^{ks}-F_k^s \sqrt{5}^{s})-(\alpha^m-F_m\sqrt{5}) 5^{(s-1)/2}\right| \\
    & < \left|\beta\right|^m \; 5^{(s-1)/2} + 2^s \sum_{k=n}^{n+d}\alpha^{ks-2k} \\
    & <  \left|\beta\right|^m \;   5^{(s-1)/2} + 2^s\; \alpha^{(n+d)(s-2)}\frac{\alpha^{s-2}}{\alpha^{s-2}-1} \\
    & <  \left|\beta\right|^m  \;  5^{(s-1)/2} + 4.2^s \cdot \alpha^{(n+d)(s-2)}.
\end{align*}
 In order to estimate the term $\left|\beta\right|^m \;  5^{(s-1)/2}$, we recall that $m > s(n+d-2)+1$. Since $\big|\beta\big|<1$, we obtain that
\begin{align*}
     \left|\beta\right|^m \;  5^{(s-1)/2} 
     & <  \left|\beta\right|^{s(n+d-2)+1} \;  5^{(s-1)/2} \\
     & =  (\left|\beta\right|^2\sqrt{5})^{s-1}\; \big|\beta\big|^{(n+d-4)s+3} \\ 
     &<  \frac{\alpha^{4s-3}}{\alpha^{s(n+d)}} \\
     & <  2^{s+2}\; \alpha^{(n+d)(s-2)}.
 \end{align*}
Hence,
\[\left|\sum_{k=n}^{m+d}\alpha^{ks}-\alpha^{m}\;  5^{(s-1)/2}\right| < 2^{s+3}\cdot\alpha^{(n+d)(s-2)}. \]
Dividing both sides of the last inequality by $\alpha^{n+d}$ yields
 \begin{equation} \label{desigfinal}
    \left|\sum_{k=0}^{d}\alpha^{-ks}-\alpha^{m-ns-ds} \; 5^{(s-1)/2}\right|<2^{s+3} \cdot \alpha^{-2(n+d)}. 
 \end{equation}
 On the other hand, we can rewrite the left-hand side of the above inequality as
 \begin{align*}
      \left|\sum_{k=0}^{d}\alpha^{-ks}-\alpha^{m-ns-ds} \; 5^{(s-1)/2}\right|
      & =\left|\frac{\alpha^s-\alpha^{-s(d+1)}}{\alpha^s-1}-\alpha^{m-ns-ds} \; 5^{(s-1)/2}\right| \\
      & > \left|\frac{\alpha^s}{\alpha^s-1}-\alpha^{m-ns-ds} \; 5^{(s-1)/2}\right|-\frac{1}{\alpha^s(\alpha^s-1)} . 
  \end{align*}
Thus, combining the last inequality with \eqref{desigfinal} yields
\[ \alpha^{-2(m+d)}>\frac{1}{2^{s+3}}\left(\left|\frac{\alpha^s}{\alpha^s-1}-\alpha^{m-ns-ds} \cdot 5^{(s-1)/2}\right|-\frac{1}{\alpha^s(\alpha^s-1)}\right).
\]
Once more, using that 
\[ -2s+2 \le m-ns-ds \le -s+2\]
we can finally establish that,
    \[\alpha^{-2(m+d)}>\min_{\substack{ 3 \le s \le 8 \\ -2s+2 \le t \le -s+2}}\left\{\frac{1}{2^{s+3}}\left(\left|\frac{\alpha^s}{\alpha^s-1}-\alpha^{m-ns-ds} \; 5^{(s-1)/2}\right|-\frac{1}{\alpha^s(\alpha^s-1)}\right)\right\}>\frac{0,018}{2^{12}}.
    \]
Therefore,  $m+d<26$. A computational verification revealed that there are no solutions in this case and, thus, the theorem is proved.
\end{proof}

In the Appendix, we explain the computational method used in the proof of the previous theorem.

%%%%%%%%%%%%%%%%%%%%%%%%%%%%%%%%%%%%%%%%%%%%%%%%%%%%%%%%%%%%%%
%%%%%%%%%%%%%%%%%%%%%%%%%%%%%%%%%%%%%%%%%%%%%%%%%%%%%%%%%%%%%%
%%%%%%%%%%%%%%%%%%%%%%%%%%%%%%%%%%%%%%%%%%%%%%%%%%%%%%%%%%%%%%
%%%%%%%%%%%%%%%%%%%%%%%%%%%%%%%%%%%%%%%%%%%%%%%%%%%%%%%%%%%%%%
%%%%%%%%%%%%%%%%%%%%%%%%%%%%%%%%%%%%%%%%%%%%%%%%%%%%%%%%%%%%%%
%%%%%%%%%%%%%%%%%%%%%%%%%%%%%%%%%%%%%%%%%%%%%%%%%%%%%%%%%%%%%%

\section*{Appendix}

In what follows, we briefly describe the computational method applied to solve the problems addressed in Theorems \ref{thm-main2} and \ref{thm-main3}. We explain in more detail the method used in Theorem \ref{thm-main2}. However, the script for Theorem \ref{thm-main3} works in a completely analogous manner. The scripts were mostly written in C language.  

Regarding computational verifications, the main challenge with Diophantine equations involving Fibonacci numbers is that the Fibonacci sequence grows exponentially. This makes the execution time for verifications, if not infeasible,  very large.  To perform the computational tests required for the aforementioned theorems, we developed a computational method whose main goal is to address the problem of the exponential growth of the Fibonacci sequence.

The spirit of our computational verifications is represented in the following example. Consider the prime number $p = 3010349 \in \Z$. The period of the Fibonacci sequence modulo the prime $p$ is $62$. Thus, among the more than $3$ million values a number can assume $\pmod{p}$ modulo $p$, at most $62$ are congruent to some Fibonacci number.

We first chose a set of prime numbers $\{p_1, \ldots, p_n\} \subseteq \Z$ that are large but have a small Pisano period.
After selecting these primes, we calculated the first $200$ terms of the Fibonacci sequence modulo these primes. Since the Fibonacci sequence modulo any number is bounded, this process is almost immediate.

To check whether a triple $(n, d, s)$ satisfies 
\[ F_n^s + F_{n+d}^s = F_m,\]
the Diophantine equation from Theorem \ref{thm-main2}, the script first tests whether the expression 
\[ F_n^s + F_{n+d}^s \equiv F_m \pmod{p_1} \] 
is verified for some Fibonacci number modulo $p_1$. If this is not the case, the triple is discarded. Otherwise, the script tests whether the expression 
\[ F_n^s + F_{n+d}^s \equiv F_m \pmod{p_2} \] 
is congruent to some Fibonacci number modulo $p_2$. If this is not the case, the triple is discarded, and so on successively. 

In this sense, the prime numbers act as filters for the possible solutions to the equation 
\[ F_n^s + F_{n+d}^s = F_m.\] 
When designing the script, we must select a sufficiently large number of primes such that the set of unfiltered triples is empty or very small.

In the script verification of Theorem \ref{thm-main2} and \ref{thm-main3}, we considered the following four prime numbers: $p_1 = 39161$ (with period $110$), $p_2 = 28657$ (with period $92$), $p_3 = 9349$ (with period $38$), and $p_4 = 9901$ (with period $66$). The script showed to be very efficient, taking only $35$ seconds to solve the equation when $n + d < 168$ and $3 \leq s \leq \frac{58057}{n + d - 1}$. For comparison, in \cite{luca-oyono-11}, the authors mention a script that took approximately one hour to solve the same problem for $n \leq 150$ and $d = 1$.

%%%%%%%%%%%%%%%%%%%%%%%%%%%%%%%%%%%%%%%%%%%%%%%%%%%%%%%%%%%%%%
%%%%%%%%%%%%%%%%%%%%%%%%%%%%%%%%%%%%%%%%%%%%%%%%%%%%%%%%%%%%%%
%%%%%%%%%%%%%%%%%%%%%%%%%%%%%%%%%%%%%%%%%%%%%%%%%%%%%%%%%%%%%%
%%%%%%%%%%%%%%%%%%%%%%%%%%%%%%%%%%%%%%%%%%%%%%%%%%%%%%%%%%%%%%
%%%%%%%%%%%%%%%%%%%%%%%%%%%%%%%%%%%%%%%%%%%%%%%%%%%%%%%%%%%%%%
%%%%%%%%%%%%%%%%%%%%%%%%%%%%%%%%%%%%%%%%%%%%%%%%%%%%%%%%%%%%%%

\section*{Acknowledgement}
This is the second article resulting from the research project \textit{Fibonacci Journey}. The project had its beginning through the research program \textit{Jornadas de Pesquisa em Matemática do ICMC 2024}, held at ICMC-USP. The authors are very grateful for the hospitality and support of ICMC-USP, as well as to the organizers of this remarkable project. This work was financed, in part, by the São Paulo Research Foundation (FAPESP), Brasil. Process Numbers 2013/07375-0 and 2022/09476-7.

%%%%%%%%%%%%%%%%%%%%%%%%%%%%%%%%%%%%%%%%%%%%%%%%%%%%%%%%%%%%%%
%%%%%%%%%%%%%%%%%%%%%%%%%%%%%%%%%%%%%%%%%%%%%%%%%%%%%%%%%%%%%%
%%%%%%%%%%%%%%%%%%%%%%%%%%%%%%%%%%%%%%%%%%%%%%%%%%%%%%%%%%%%%%
%%%%%%%%%%%%%%%%%%%%%%%%%%%%%%%%%%%%%%%%%%%%%%%%%%%%%%%%%%%%%%
%%%%%%%%%%%%%%%%%%%%%%%%%%%%%%%%%%%%%%%%%%%%%%%%%%%%%%%%%%%%%%
%%%%%%%%%%%%%%%%%%%%%%%%%%%%%%%%%%%%%%%%%%%%%%%%%%%%%%%%%%%%%%

%\bibliographystyle{alpha}
%\bibliography{fibonacci2}  

\end{document}